\documentclass[10pt]{article}

\usepackage{latexsym,color,amsmath,amsthm,amssymb,amscd,amsfonts,setspace}
\usepackage{dsfont}
\usepackage{yfonts}
\usepackage{mathtools}
\usepackage{microtype}
\mathtoolsset{centercolon}

\newtheorem{theorem}{Theorem}

\newtheorem{lemma}{Lemma}

\newtheorem{proposition}{Proposition}
\newtheorem{definition}{Definition}

\newtheorem{remark}{Remark}

\def\R{{\mathbb R}}

\def\vol{{\rm vol}}

\def\1{\mathds{1}}

\def\be{\begin{equation}}
\def\ee{\end{equation}}
\def\Hess{{\nabla^2}}

\DeclarePairedDelimiter\norm{\lVert}{\rVert}

\makeatletter
\newcommand{\leqnomode}{\tagsleft@true}
\newcommand{\reqnomode}{\tagsleft@false}
\makeatother

\makeatletter
\let\@fnsymbol\@alph
\makeatother

\begin{document}

\title{\bf Weighted floating functions and\\ weighted functional affine surface areas}

\author{Carsten Sch\"utt\thanks{Universit\"at Kiel. Email: schuett@math.uni-kiel.de}, Christoph Th\"ale\thanks{Ruhr University Bochum, Germany. Email: christoph.thaele@rub.de}, Nicola Turchi\thanks{University of Milano-Bicocca. Email: nicola.turchi@unimib.it}, Elisabeth M.\ Werner\thanks{Case Western Reserve University. Email: elisabeth.werner@case.edu}}

\date{}

 \maketitle

\begin{abstract}
The purpose of this paper is to introduce the new concept of weighted floating functions associated with log concave or $s$-concave functions. This leads to new notions of weighted functional affine surface areas. Their relation to more traditional versions of functional affine surface areas as well as to the classical affine surface areas for convex bodies is discussed in detail.
\end{abstract}

\begin{spacing}{0}
{\footnotesize \tableofcontents}
\end{spacing}

\section {Introduction}

The affine surface area plays a central role in the affine geometry of convex bodies. For convex bodies in the plane or in $3$-dimensional space, this geometric functional has been introduced by Blaschke \cite{Blaschke:1923}. For general space dimensions, if $K\subset\R^n$ is a convex body with boundary of differentiability class $C^2$, its affine surface area is defined as
$$
\operatorname{as}(K) = \int_{\partial K}\kappa_K(z)^{1\over n+1}\,d\mu_K(z),
$$
where $\mu_K$ stands for the surface measure on the boundary $\partial K$ of $K$ and $\kappa_K(z)$ for the Gaussian curvature at a boundary point $z$. We remark that the definition can be extended to all convex bodies, see, e.g.,  \cite{Leichtweiss, SchuettWerner1990, Werner94}. The affine surface area has a number of important properties:
\begin{itemize}
\item[(i)] it is equiaffine invariant in the sense that $\operatorname{as}(\varrho(K))=\operatorname{as}(K)$ for any volume preserving affine transformation $\varrho:\R^n\to\R^n$;
\item[(ii)] it is upper semi-continuous, which means that $\operatorname{as}(K)\geq\limsup\limits_{n\to\infty}\operatorname{as}(K_n)$ for any sequence of convex bodies $K_n$, which converge in Hausdorff distance to a convex body $K\subset\R^n$;
\item[(iii)] it is a valuation, that is, $\operatorname{as}(K\cup L)=\operatorname{as}(K)+\operatorname{as}(L)-\operatorname{as}(K\cap L)$ for any two convex bodies $K,L\subset\R^n$ for which also the union $K\cup L$ is a convex body.
\end{itemize}
The point (i) goes back to Blaschke, (ii) was proved by Lutwak \cite{Lutwak1991} and (iii) by Sch\"utt \cite{Schutt1993}.
It was shown in \cite{LudwigReitzner99} that -- aside from the Euler characteristic and the volume -- the affine surface area is up to constants the only real-valued functional on the space of convex bodies satisfying these three properties. Its affine invariance also explains the central role of the affine surface area in the theory of affine inequalities. For example, the affine isoperimetric inequality says that 
$$
\Big({\operatorname{as}(K)\over\operatorname{as}(B_2^n)}\Big)^{1\over n-1} \leq \Big({\vol_n(K)\over\vol_{n}(B_2^n)}\Big)^{1\over n+1}
$$
for all convex bodies $K\subset\R^n$ with centroid at the origin and where $B_2^n$ is the $n$-dimensional Euclidean unit ball. Equality holds in this inequality if and only if $K$ is an ellipsoid.

During the last decades there has been a growing interest in extending geometric inequalities for convex bodies to inequalities for functions, which typically reduce to their geometric counterparts if appropriate indicator or gauge functions are chosen. A prominent example in this direction is the Pr\'ek\'opa-Leindler inequality, which can be regarded as the functional version of the classical Brunn-Minkowski inequality for convex bodies. Also the affine isoperimetric inequality has a functional counterpart, which has been developed in \cite{AAKSW}. Namely, if $\psi:\R^n\to\R$ is a convex function with $\int_{\R^n}xe^{-\psi(x)}\,dx=0$ one defines its affine surface area by
$$
\operatorname{as}(\psi) = \int_{\R^n}e^{-(\psi(x)+\langle x,\nabla\psi(x)\rangle)}\,\det\nabla^2\psi(x)\,dx
$$
where $\langle\,\cdot\,,\,\cdot\rangle$ is the standard scalar product on $\R^n$ and $\nabla^2\psi$ the Hessian of $\psi$. Then the functional affine isoperimetric inequality says that
$$
{\operatorname{as}(\psi)\over\operatorname{as}(\psi_{\|\,\cdot\,\|})} \leq {\int_{\R^n}e^{-\psi_{\|\,\cdot\,\|}(x)}\,dx\over\int_{\R^n}e^{-\psi(x)}\,dx}\qquad\text{with}\qquad \psi_{\|\,\cdot\,\|}(x)={\|x\|^2\over 2}
$$
and where $\|\,\cdot\,\|$ is the Euclidean norm on  $\R^n$. We remark that equality holds if and only if $\psi(x)=\langle Ax,x\rangle+b$ for some positive definite matrix $A$ and a scalar $b\in\R$.

The affine surface area of a convex body $K\subset\R^n$ is closely related to what is known as the floating body $K_\delta$. For sufficiently small $\delta>0$ the latter is the intersection of all halfspaces which cut off caps of volume at most $\delta$ from $K$, see \eqref{schwimm} below. It was shown in \cite{SchuettWerner1990} that
\begin{equation}\label{eq:5-2-24B}
\lim_{\delta\to 0}{\vol_n(K)-\vol_n(K_\delta)\over\delta^{2\over n+1}} = {1\over 2}\Big({n+1\over\vol_{n-1}(B_2^{n-1})}\Big)^{2\over n+1}\,\operatorname{as}(K).
\end{equation}
It is this result together with the observation that random polytopes are well approximated by suitable floating bodies that makes the affine surface area appear in problems of stochastic geometry and the theory of approximation of convex bodies by polytopes. More precisely, for a convex body $K\subset\R^n$, $N\geq n+1$ and a sequence of independent and uniformly distributed random points $X_1,X_2,\ldots$ in $K$ one considers the random convex hull $[K]_N={\rm conv}\{X_1,\ldots,X_N\}$. Then it turns out that $[K]_N$ roughly behaves like the floating body $K_{1/N}$ and that
\begin{equation}\label{random}
\lim_{N\to\infty}{\vol_n(K)-\mathbb{E}\vol_n([K]_N)\over({\vol_n(K)\over N})^{2\over n+1}} = {1\over 2}\Big({n+1\over\vol_{n-1}(B_2^{n-1})}\Big)^{2\over n+1}\,\operatorname{as}(K),see
\end{equation}
\cite{Barany:1992, Schuett1994}. 
This connection has also been the starting point of some more recent developments in the theory of random polytopes in spherical and hyperbolic space, see \cite{BesauRosenThaele,BesauThaele}. However, to deal with such situations, an additional weight function is required. This has motivated the study of weighted floating bodies in \cite{BLW:2016, Werner2002}. The purpose of the present paper is to extend the notion of weighted floating bodies to a functional level, extending thereby the earlier work \cite{LiSchuettWerner} on (unweighted) floating functions. We do this in two closely related set-ups by considering
\begin{itemize}
\item[(i)]convex or log concave functions;
\item[(ii)] $s$-concave functions.
\end{itemize}
The concept of  weighted floating functions in the first case is introduced by considering what we call weighted floating sets of the (unbounded) epigraphs of these functions. This leads to the new notion of weighted functional affine surface areas for convex or log concave functions.  In contrast, we use the by now classical concept of weighted floating bodies from \cite{BLW:2016,Werner2002} to associate floating functions to $s$-concave functions by means of an auxiliary convex body that is connected with such a function. The latter has been introduced in \cite{KlartagMarginals} for the study of geometric inequalities. This approach leads to the new notion of weighted functional affine surface areas for $s$-concave functions. The main motivation for calling the new quantities  affine surface areas is that we are able to prove analogues to \eqref{eq:5-2-24B} for each of the functionals we consider. This is the content of our main results, Theorems \ref{theo:f-deltafloat2}, \ref{theo:f-deltafloat1} and \ref{thm:sconcave} below.

The remaining parts of this text are structured as follows. After setting up some conventions and notation in Section \ref{sec:Conventions} we introduce weighted floating sets and weighted floating functions in Section \ref{WFS}. We do this separately for (i) convex or log concave function as well as for (ii) $s$-concave functions. The main theorems of this paper are the content of Section \ref{sec:MainResults}, where we again separate the cases (i) and (ii). All proofs are collected in the final Section \ref{proofs}.

\section{Conventions on notation}\label{sec:Conventions}

Throughout the paper we will use the following notation. We will be working in  $n$-dimensional Euclidean space $\R^n$ with scalar product $\langle\,\cdot\,,\,\cdot\,\rangle$ and norm $\|\,\cdot\,\|$. The $n$-dimensional Lebesgue measure will be indicated by $\vol_{n}(\,\cdot\,)$.
We denote by $B^{n}_2(x,r)$ the $n$-dimensional closed Euclidean  ball centered at $x\in\R^n$ with radius $r>0$. We write in short $B^{n}_2=B^{n}_2(0,1)$ for the  Euclidean unit ball centered at the origin $0$ and $S^{n-1} = \partial B^n_2$ for the $(n-1)$-dimensional unit sphere in $\R^n$.
A convex body $K$ in $\mathbb{R}^n$ is a compact convex subset of $\mathbb{R}^n$ with non-empty interior  $\text{int}(K)$. We denote  by  $\partial K$  the boundary  of $K$ and by $\operatorname{int}(K)$ its interior.
 The hyperplane passing through a point $x\in \mathbb{R}^n$ and which is orthogonal to a vector $u \in S^{n-1}$ is  denoted by $H(x, u)$. We write $H$ in short if there is no confusion and then 
 $H^+$ and $H^-$ stand for the two closed halfspaces determined by $H$. Finally,  $c, \alpha_1, \alpha_2,\ldots,\beta_1,\beta_2,\ldots\in(0,\infty)$ denote absolute constants that may change from line to line.

\section{Weighted floating sets and functions} \label{WFS}
\subsection{Weighted floating sets}

The concept of the so-called floating body is very classical in the affine geometry of convex bodies and goes back to Dupin and Blaschke in dimension $n=2$ and $n=3$, see \cite{Blaschke:1923}. We start by recalling the general definition, which was given independently by B\'ar\'any and Larman \cite{BaranyLarman:1988} and Sch\"utt and Werner \cite{ SchuettWerner1990}. For a unit vector $u \in S^{n-1}$ and a point $x \in \mathbb{R}^n$ we denote by  $H(x, u)$, or simply by $H$,  the  hyperplane through $x$ and orthogonal to $u$, that is
 $$H=H(x,u) =\{y \in \mathbb{R}^n: \langle y, u\rangle= \langle x, u\rangle  \}.$$ 
Then, setting $\langle x, u\rangle  = a$,   $H^+ =\{y \in \mathbb{R}^n: \langle y, u \rangle \geq  a\}$ and $H^-=\{y \in \mathbb{R}^n: \langle y, u \rangle  \leq  a\}$ are  the two closed half spaces determined by $H$. Now, let $K$ be a convex body in $\mathbb{R}^n$ and $\delta \geq 0$.   
The {\em (convex) floating body }
$K_{\delta}$  of $K$ is
the intersection of all halfspaces $H^+$ whose
bounding hyperplanes $H$ cut off  sets of volume at most $\delta$
from $K$. Formally ,
\begin{equation} \label{schwimm}
K_{\delta}=\bigcap\Big\{H^+:\vol_n(H^-\cap K) \leq \delta\Big\},
\end{equation}
see \cite{BaranyLarman:1988, SchuettWerner1990}. 
The floating body exists, i.e.,  is non-empty,  if $\delta$ is small enough. Moreover,  $K_0=K$ and $K_{\delta}\subseteq K$,  for all $\delta \geq 0$.

We mention  that floating bodies have tight connections to random polytopes. In fact, it turns out that the convex hull $[K]_N$ of $N\geq n+1$ independent and uniformly distributed random points in $K$ is `close' to the floating body $K_{1/N}$ for large $N$ in a sense explained in detail in the survey article \cite{ReitznerSurvey}. More recently, in  \cite{BLW:2016,BesauRosenThaele,BesauThaele} it was observed that random polytopes in spherical or hyperbolic space are also connected to the theory of floating bodies. In this case, however, one needs to replace the $n$-dimensional volume in \eqref{schwimm} by a weighted volume. This leads to the definition of \textit{weighted floating bodies}, a concept that has been introduced in \cite{Werner2002} (see also \cite{BesauWerner2}). Again, we let $K\subset\R^n$ be a convex body and $\delta\geq 0$. In addition, we let $\Phi:  \mathbb{R}^{n} \to \mathbb{R}_+$ be an integrable function.  Then the weighted floating body $K^\Phi_\delta$ (with respect to the weight function $\Phi$) is defined as
\begin{equation} \label{weighted-schwimm}
K^\Phi_\delta=  \bigcap \left\{ H^+ :  \int _{H^-\cap K} \Phi(x) \, dx \leq \delta \right\}.
\end{equation}
To extend the concept of weighted floating bodies from convex bodies to a functional level, it is necessary to extend the definition of weighted floating bodies to unbounded closed convex sets $C\subset\R^n$ having non-empty interior. For $x\in\partial C$ we denote by $N_C(x)$ the set of all outer unit normal vectors at $x$. We remark that by a theorem of Rademacher \cite{Rademacher} that $N_C(x)$  consists of a single vector at $ \mu_C$-a.e.\ points $x$ on $ C$, where $ \mu_C$ stands for the surface measure on $\partial C$. Then, for a continuous function $\Phi: \mathbb{R}^{n} \to \mathbb{R}_+$ and $\delta\geq 0$ we put
\begin{align*}
	C^\Phi_{\delta}= \bigcap \left\{H\left(x-t N_C(x), N_C(x)\right)^+ :  \int _{H(x -t N_C(x), N_C(x))^-\cap C} \Phi(x) \, dx \leq \delta \right\}.
\end{align*}
It is clear that $C^\Phi_\delta $ is a closed convex subset of $C$. If   $\Phi \equiv 1$, the definition coincides with the one introduced  in \cite{LiSchuettWerner}.
In that case we write simply write $C_\delta$ instead of $C^{\Phi\equiv 1}_\delta$.
While for a convex body $K$, $K_\delta$ is a proper subset of $K$ if $\delta >0$, it is now possible that $C_\delta=C$ for  $\delta >0$. This happens, for example, if $C$ is a halfspace.

\begin{remark}
As anticipated above, we will consider the definition of weighted floating sets in a functional setting,  and therefore we want to obtain an unbounded floating set. Thus, defining the floating set as $C^\Phi_\delta=  \bigcap \left\{ H^+ :  \int _{H^-\cap K} \Phi(x) \, dx \leq \delta \right\}$ for a integrable function $\Phi: \mathbb{R}^{n} \to \mathbb{R}_+$ may lead to a bounded $C^\Phi_\delta$,
which is not desirable in our context.
\end{remark}

\subsection{Convex and log concave functions and their floating functions}

Let $\psi: \R^n \rightarrow \R $
 be a  convex function.   We always consider in this paper convex  functions 
$\psi$ satisfying the assumption that  $0 <  \int _{\mathbb R^{n}} e^{-\psi(x)} dx < \infty$. 
We denote the set of such functions by $\text{Con}(\mathbb{R}^n)$.  We also recall that a function 
$f: \mathbb{R}^n \rightarrow \mathbb{R}_+$ is   log concave, if it is of the form $f= \exp(-\psi )$ for some convex function  $\psi: \R^n \rightarrow\R $.

In the general case, when $\psi$ is neither  smooth nor strictly convex, 
the gradient of $\psi$, denoted  by $\nabla \psi$, exists almost everywhere by Rademacher's theorem \cite{Rademacher}. Furthermore,  a theorem of Alexandrov \cite{Alexandroff} and Busemann and Feller \cite{Buse-Feller} guarantees the existence of the (generalized) Hessian, denoted  by
$\nabla^2 \psi$, almost everywhere in $\mathbb R^{n}$. The Hessian  is a quadratic form on $\mathbb{R}^n$, and if $\psi$
is a convex function, a Taylor formula  as in the $C^2$ case  holds for almost every $x \in \mathbb{R}^n$ (see \cite{SchuettWernerBook}). Formally, we have that 
$$
\psi( x + y) = \psi (x) + \langle \nabla   \psi(x),  y  \rangle + \frac{1}{2}  \langle \nabla^2 \psi(x) (y), y \rangle + o( \|y\|^2), 
$$
when $y \to 0$.

 Let $\psi$ be in $\text{Con}(\mathbb{R}^n)$ and denote by  
$$
\operatorname{epi}(\psi) = \{ z=(x,y) \in \mathbb{R}^n \times \mathbb{R}:   y \geq \psi(x)\}
$$
the  epigraph of $\psi$.
Then $\operatorname{epi}(\psi)$  is a  closed  convex subset of $\mathbb{R}^{n+1}$. We denote for short $N_\psi(z)= N_{\operatorname{epi}(\psi)}(z)$ for $z\in\R^{n+1}$. By the  discussion of the previous section,  for sufficiently small $\delta$ and a continuous weight function  $\Phi: \mathbb{R}^{n+1} \to \mathbb{R}_+$, the floating sets $\operatorname{epi}(\psi)^\Phi_\delta$ are given by
\begin{equation*}
\operatorname{epi}(\psi)^\Phi_\delta = 
 \bigcap \Big\{H\left(x-t \, N_{\psi}(z), N_{\psi}(z)\right)^+ :  \int _{H(x -t \, N_{\psi}(z), N_{\psi}(z))^-\, \cap\,  \operatorname{epi}(\psi)} \Phi(z) \, dz \leq \delta \Big\}.
\end{equation*}
 It is easy to see that there exists a unique convex function $\psi^\Phi_\delta: \mathbb{R}^n \rightarrow \mathbb{R}$ 
such that $(\operatorname{epi}(\psi))^\Phi_\delta =\operatorname{epi}(\psi^\Phi_\delta)$.  
This leads to the definitions of weighted  floating functions  for convex and  log concave  functions.

\begin{definition} 
Let $\psi$ be in $\text{Con}(\mathbb{R}^n)$, $\delta >0$ and  $\Phi: \mathbb{R}^{n+1} \to \mathbb{R}_+$ be a continuous function.
\begin{itemize}
\item[(i)] The  weighted floating function  of $\psi$ is defined to be the function $\psi^\Phi_\delta$ satisfying
\begin{equation}\label{flotfunct}
	(\operatorname{epi}(\psi))^\Phi_\delta = \operatorname{epi}\left(\psi^\Phi_\delta\right).
\end{equation} 

\item[(ii)] Let $f(x)= \exp(-\psi (x) )$ be a log concave function. The weighted  floating function $f^\Phi_\delta$ of $f$ is defined as
\begin{equation}\label{flotlog}
	f^\Phi_\delta (x) = \exp\left(-\psi^\Phi_\delta (x) \right).
\end{equation}
\end{itemize}
\end{definition}
\par
 
For a log concave function $f$, a natural weight  function is the exponential weight function $\Phi(z)=\Phi_e(z)=\Phi((x,s)) =e^{-s}$. We then write in short $\psi_\delta^e$ and $f_\delta^e$, where we identify a point $z\in\R^{n+1}$ with a pair $(x,s)\in\R^n\times\R$, representing a spatial and a height coordinate of $z$.
When $\Phi \equiv 1$, we  write $\psi_\delta$ and $f_\delta$. Note also that when $\psi$ is affine,   $\psi_\delta=\psi$ and for $f=e^{-\psi}$ we have that $f_\delta=f$.

The next proposition collects  properties   of the floating set $\operatorname{epi}(\psi)^\Phi_\delta$ which we will need later. The proofs follow   from the ones given for the floating body in \cite{SchuettWerner1994}
and \cite{SchuettWernerBook} by truncating  the unbounded convex set $\operatorname{epi}(\psi)$ appropriately, applying the results for convex bodies and then and by a limiting argument. To present the result,
for $A \subset \mathbb{R}^n$, we call the point $$g^\Phi_A=\frac{\int_A x\,  \Phi(x) \, dx}{\int_A   \Phi(x) \, dx}$$ the $\Phi$-barycenter of $A$. 

\begin{proposition}\label{properties}
Let $\psi\in\text{Con}(\mathbb{R}^n)$ and  $\Phi: \mathbb{R}^{n} \to \mathbb{R}_+$ be a continuous function.
\begin{itemize}
\item[(i)] For all \(\delta\) such that \( \operatorname{epi}(\psi)^\Phi_\delta \neq \emptyset \) and all \( z_\delta\in \partial(\operatorname{epi}(\psi)^\Phi_\delta) \, \cap\operatorname{int} (\operatorname{epi}(\psi))\) there exists a support hyperplane \(H\) at \( z_\delta\) to \( \operatorname{epi}(\psi)^\Phi_\delta \) such that \(\delta= \int _{\operatorname{epi}(\psi)\cap H^-} \Phi(z) dz\).

\item[(ii)] A supporting hyperplane $H$ of $\operatorname{epi}(\psi)^\Phi_\delta$ that cuts of a set of $\Phi$-volume $\delta$ from $\operatorname{epi}(\psi)$ touches 
$\operatorname{epi}(\psi)^\Phi_\delta$ in exactly one point, the $\Phi$-barycenter of $\operatorname{epi}(\psi)\cap H$. 

\item[(iii)] The set $\operatorname{epi}(\psi)^\Phi_\delta$ is strictly convex.
\end{itemize}
\end{proposition}

\begin{remark}
The proposition also holds for a bounded strictly convex subset $C \subset \mathbb{R}^n$.
\end{remark}
 
\subsection{$s$-concave functions and their floating functions}

The purpose of this section is to introduce another notion of weighted floating functions, which is essentially based on the classical notion of a weighted floating body. It works for functions which satisfy stronger concavity assumptions in comparison to the log concavity property we assumed in the previous section. To introduce the necessary concepts, let \(n,s\in \mathbb{N}\). A function \(f\colon\R^n\to\mathbb{R}_+\) is called \emph{s-concave} if its support \(\mathrm{supp}(f)\) is a compact convex subset of $\R^n$ with non-empty interior and if \(f^{1/s}\) is concave on its support. Following \cite{KlartagMarginals}, to such a function \(f\) we associate the set
\begin{equation}\label{eq:Kfs}
K_f^s\coloneqq\{(x,y)\in\R^n\times\R^s:x\in\mathrm{supp}(f),\norm{y}\le f^{1/s}(x)\},
\end{equation}
where \(\|\,\cdot\,\|\)  the Euclidean norm on \(\R^s\). Note that \(K_f^s\) is a convex body of revolution and its \((n+s)\text{-dimensional}\) volume \(\vol(K_f^s)\) is related to the integral \(\int_{\R^n}f(x)\, dx\) via the identity
\[
\vol_{n+s}(K_f^s)=\int_{\R^n}\vol_s(B_2^s) \bigl(f^{1/s}(x)\bigr)^s\, d x=\vol_s(B_2^s)\int_{\R^n}f(x)\, d x.
\]
Let \(\Phi\colon \R^{n+s}\supset K_f^{s}\to\mathbb{R}_+\) be a continuous function such that \(\Phi(x,y)=\phi(x,\norm{y})\) for some \(\phi\colon \R^{n+1}\to\R_+\). 

For \(\delta>0\) we denote by \(K_f^s(\Phi,\delta)\) the \((\Phi,\delta)-\)weighted floating body of \(K_f^s\), which is indeed a convex body if \(\delta\) is sufficiently small, see \cite{BLW:2016,Werner2002}. Moreover, since \(K_f^s\) is a body of revolution and thanks to the rotational symmetry of \(\Phi\) in the last $s$ coordinates, the set of all hyperplanes that cut a cap \(C\) off from \(K_f^s\), such that \(\int_C \Phi(x)\,dx=\delta\), is invariant under transformations of the form \(\R^n\times \R^s\ni (x,y)\mapsto(x,Ry)\) with \(R\in \mathbb{O}(s)\), the group of orthogonal transformations acting in the $\R^s$-coordinate. As a consequence, \(K_f^s(\Phi,\delta)\), being the intersection of half-spaces determined by such hyperplanes, is a convex body of revolution itself.

By Brunn's concavity principle there exists a concave function \(g^\Phi_\delta \colon\R^n\to\mathbb{R}_+\) such that
\[
K_f^s(\Phi,\delta)=\{(x,y)\in\R^n\times\R^s:x\in\mathrm{supp}(g^\Phi_\delta),\norm{y}\le g^\Phi_\delta (x)\},
\]
where \( \mathrm{supp}(g^\Phi_\delta)\coloneqq\mathrm{supp}(f)\cap K_f^s(\phi,\delta)\). Indeed, let \(e_1,\ldots, e_n,e_{n+1},\ldots, e_{n+s}\) be the standard orthonormal basis of \(\R^n\times \R^s\) and \(L_s\coloneqq\mathrm{span}\{e_{n+1},\ldots, e_{n+s}\}\) be the $s$-dimensional linear subspace generated by the last $s$ unit vectors. Then the function \(r^\Phi_\delta \colon\R^n\to \mathbb{R}_+\) defined by
\[
r^\Phi_\delta (x)\coloneqq \vol_s(K_f^s(\Phi,\delta)\cap(x+L_s))^{1/s}
\]
is concave on its support according to \cite[Theorem 1.2.1]{AGAVol1} and we can choose
\[
g^\Phi_\delta (x)\coloneqq \vol_s(B_2^s)^{-1/s} r^\Phi_\delta (x).
\]
As a consequence, the function \(f^\Phi_\delta(x)\coloneqq g^\Phi_\delta(x)^s\) is  \(s\text{-concave}\).

\begin{definition}
Let $\delta>0$ and \(\Phi\colon \R^{n+s}\supset K_f^{s}\to\mathbb{R}_+\) be a continuous function such that \(\Phi(x,y)=\phi(x,\norm{y})\) for some \(\phi\colon \R^{n+1}\to\mathbb{R}_+\). For an $s$-concave function $f:\R^n\to\mathbb{R}_+$ we call $f^\Phi_\delta$ the $s$-concave $\Phi$-weighted floating function of $f$.
\end{definition}

\section {Main Theorems and their consequences}\label{sec:MainResults}

\subsection{The case of log concave functions}\label{sec:ConvexLogConcaveResults}

We start by presenting our main results for log concave functions. Throughout the paper we put
\begin{equation}\label{cn+1}
c_{n+1} =  \frac{1}{2} \left(\frac{n+2}{\vol_{n} (B^{n}_2)}\right)^\frac{2}{n+2}
\end{equation}
and can now formulate our first result. It deals with weight functions which are uniformly bounded away from zero.
 
\begin{theorem}\label{theo:f-deltafloat2}
 Let $\psi$ be in $\text{Con}(\mathbb{R}^n)$ and $f(x)= e^{-\psi(x)}$.  Let $\eta >0$ and 
let $\Phi: \mathbb{R}^{n+1} \to [\eta, \infty)$ be a continuous function.
Then,
$$  
\lim _{\delta \rightarrow 0} \frac{ \int _{\mathbb R^{n}}(f(x)  - f_\delta ^\Phi)(x)  \  dx } {\delta^{2/(n+2)}} =  c_{n+1} \int_{\mathbb R^{n}} \left(\det\left(\nabla^2 \psi (x) \right)\right)^\frac{1}{n+2} \  e^{-\psi(x)} \, \Phi((x, \psi(x)))^{-\frac{2}{n+2}} dx.
$$
\end{theorem}
\vskip 2mm
 
The next result concerns the case of the exponential weight function $\Phi(z)=\Phi_e(z)=e^{-s}$, $z=(x,s)\in\R^{n+1}$. We cannot obtain it from the previous theorem as the assumption $\eta>0$ there is not satisfied. A separate proof is needed to handle this case.

\begin{theorem}\label{theo:f-deltafloat1}
Let $\psi$ be in $\text{Con}(\mathbb{R}^n)$ and $f(x)= e^{-\psi(x)}$.
Then,
\begin{equation*}\label{asaf} 
\lim _{\delta \rightarrow 0} \frac{ \int _{\mathbb R^{n}}(f(x)  - f_{\delta}^{e}(x) )  \  dx } {\delta^{2/(n+2)}} =  c_{n+1} \int_{\mathbb R^{n}} \left(\det\left( \nabla^2 \psi (x) \right)\right)^\frac{1}{n+2}  e^{-\frac{n}{n+2} \psi(x)}\, dx.
\end{equation*}
\end{theorem}

\begin{remark}
If the determinant of the Hessian of $\psi$ is zero almost everywhere, the right hand terms in the two theorems above are zero as well.  This is in particular the case when $\psi$ is an affine or piecewise affine function. When $\psi$ is affine this follows immediately because, as noted above,  $\psi^\Phi_\delta=\psi$ and $f^\Phi_\delta=f$. When $\psi$ is piecewise affine, then the difference $\int _{\mathbb R^{n}}(f  - f_\delta ^\Phi)  \,  dx$ is of order $\frac{\delta}{(\log \delta)^n}$
and thus the left hand sides of the theorems are also $0$.
It is therefore enough to consider  functions that  are not (piecewise) affine.
\end{remark}

We now show that under the assumptions of the theorems, the right hand side integrals are finite. To do so, we use the notion of a rolling function which  was  introduced for convex bodies in \cite{SchuettWerner1990} and extended to not necessarily bounded  convex sets in \cite{LiSchuettWerner}.
Let $C$ be a closed convex set in $\R^n$  and let $z \in \partial C$ be such that $N_{C} (z)$, the outer normal vector of $C$ at $z$,  is unique.  We put $r_C(z)$ to be the radius of the largest 
Euclidean ball contained in $C$ that touches $C$ in $z$, 
\begin{equation} \label{rC}
r_C(z) = \max\{\rho: B^{n}_2 (z - \rho N_{C} (z), \rho ) \subset C\}.
\end{equation}
If $N_C(z)$ is not unique, we put $r_C(z) =0$. 
The function $r_{C}$ is called the \textit{rolling function} of $C$.
If $C=\operatorname{epi}(\psi)$, we  will  use from now on  the notation  
\begin{equation} \label{r(x)}
r_\psi(x) = r _{\operatorname{epi}(\psi)}((x, \psi(x))).
\end{equation}
Since $\psi$ is lower semi continuous,   the epigraph of $\psi$ is a closed set.
For functions $\psi$ such that $e^{-\psi}$ is integrable,  $r_{\psi}(z)$ is bounded  and for almost every $x \in \mathbb{R}^n$, $(x, \psi(x))$ is an element of a  Euclidean ball contained in the epigraph of $\psi$.

For a convex  function $\psi\in\text{Con}(\mathbb{R}^n)$ the following formulas  hold for the (generalized) Gaussian curvature $\kappa_{\psi} (z)$  and the outer unit normal $N_{\psi} (z)$ in $z=(x, \psi(x)) \in \partial \operatorname{epi}(\psi)$, see, e.g., \cite{CFGLSW}:
\begin{align} 
\kappa_{\psi} (z) &= \frac{\det (\nabla^2 \psi(x))}{\left(1 + \|\nabla \psi (x)\|^2\right)^\frac{n+2}{2}}\label{curvature}
\intertext{and}
\langle N_{\psi} (z), e_{n+1}\rangle&=\frac{1}{(1+\|\nabla \psi (x)\|^2)^{\frac{1}{2}}}.\label{normal}
\end{align}
As $ \kappa_{\psi} (z) = \prod_{i=1}^n \frac{1}{\rho^\psi_i(z)}$, where $\rho^\psi_i(z)$, $1 \leq i \leq n$,  are the principal radii of curvature, we have for almost all $x \in \mathbb{R}^n$ that  $r_\psi(x) \leq \frac{1}{\left(\kappa_{\psi} (z)\right)^\frac{1}{n}}$ . 
With \eqref{curvature} we thus get 
$$
r_\psi(x) \leq \frac{1}{\left(\kappa_{\psi} (z)\right)^\frac{1}{n}}  = \frac {\left(1 + \|\nabla \psi (x)\|^2\right)^\frac{n+2}{2n}} {\left(\det \nabla^2 \psi(x)\right)^\frac{1}{n}}.
$$
Therefore,
\begin{eqnarray*}
\int_{\mathbb R^{n}} \left(\det\left( \nabla^2 \psi (x) \right)\right)^\frac{1}{n+2}  e^{-\frac{n}{n+2} \psi(x)}\, dx \leq 
\int_{\mathbb R^{n}} \frac{ \left(1 + \|\nabla \psi (x)\|^2\right)^\frac{1}{2}}{r_\psi(x) ^\frac{n}{n+2} }   e^{-\frac{n}{n+2} \psi(x)}\, dx  
\end{eqnarray*}
and 
\begin{align*}
&\int _{\mathbb R^{n}}\left(\det\left( \nabla^2( \psi (x))\right) \right)^\frac{1}{n+2} \, \Phi(x, \psi(x))^{-\frac{2}{n+2}}  f(x) \,   dx \\
&\qquad\leq \eta ^{-\frac{2}{n+2}}\,  \int_{\mathbb R^{n}} \frac{ \left(1 + \|\nabla \psi (x)\|^2\right)^\frac{1}{2}}{r_\psi(x) ^\frac{n}{n+2} }   f(x) \,   dx  .
\end{align*}
Now, Lemmas \ref{integral1} and  \ref{integral2} below show that the last two integrals are indeed finite. 

Next, we record two propositions, which follow 
from the lemmas needed for the proof of 
Theorems \ref{theo:f-deltafloat2} and \ref{theo:f-deltafloat1}. 
 
\begin{proposition}\label{corollary1}
Let $\psi$ be in $\text{Con}(\mathbb{R}^n)$.  Let $\eta >0$ and 
let $\Phi: \mathbb{R}^{n+1} \to [\eta, \infty)$ be a continuous function.
Then
\begin{eqnarray*}
&& \lim _{\delta \rightarrow 0} \frac{ \int_{\mathbb R^{n}} \left| \psi^\Phi_\delta(x) - \psi (x)\right| \, e^{-\psi(x)} \,  dx} {\delta^{2/(n+2)}}\\
&& \hskip 20mm = c_{n+1} \int_{\mathbb R^{n}} \left(\det\left( \nabla^2 \psi (x) \right)\right)^\frac{1}{n+2} \, \Phi((x, \psi(x)))^{-\frac{2}{n+2}}  e^{-\psi(x)} \, dx.
\end{eqnarray*}
\end{proposition}

The case of the exponential weight function $\Phi=\Phi_e$ needs a separate treatment and leads to the following result.
 
\begin{proposition}\label{corollary2}
Let $\psi$ be in $\text{Con}(\mathbb{R}^n)$. 
Then
\begin{eqnarray*}
\lim _{\delta \rightarrow 0} \frac{ \int_{\mathbb R^{n}} \left| \psi^e_\delta(x) - \psi (x)\right| \, e^{-\psi(x)} \,  dx} {\delta^{2/(n+2)}} =  
 c_{n+1} \int_{\mathbb R^{n}} \left(\det\left( \nabla^2 \psi (x) \right)\right)^\frac{1}{n+2} \,  e^{-\frac{n}{n+2} \psi(x)} \, dx.
\end{eqnarray*}
\end{proposition}

Theorem \ref{theo:f-deltafloat2} and Proposition \ref{corollary1} as well as Theorem \ref{theo:f-deltafloat1} and  Proposition \ref{corollary2} motivate the following definition.

\begin{definition}
Let $\psi$ be in $\text{Con}(\mathbb{R}^n)$.  Let $\eta >0$ and 
let $\Phi: \mathbb{R}^{n+1} \to [\eta, \infty)$ be a continuous function. Then
\begin{equation}\label{asf}
	\operatorname{as}_\Phi(f) =\operatorname{as}_\Phi(\psi) = \int_{\mathbb R^{n}} \left(\det\left( \nabla^2 \psi (x) \right)\right)^\frac{1}{n+2} \Phi((x, \psi(x)))^{-\frac{2}{n+2}} f(x) \,  d x.
\end{equation}
is called the $\Phi$-(weighted) affine surface area of  the log concave function $f=e^{-\psi}$  or the convex function $\psi$, respectively. 

For the exponential weight function $\Phi=\Phi_e$ we call 
\begin{equation}\label{asf-e}
	\operatorname{as}_{\Phi_e} (f)=\operatorname{as}_{\Phi_e}(\psi)  = \int_{\mathbb R^{n}} \left(\det\left( \nabla^2 \psi (x) \right)\right)^\frac{1}{n+2} \, e^{-\frac{n}{n+2}\psi(x) }\,  d x.
\end{equation}
the $\Phi_e$-(weighted) affine surface area of  the log concave function $f=e^{-\psi}$  or the convex function $\psi$, respectively. 
\end{definition}

Our results demonstrate that  these $\Phi$-affine surface areas can be obtained as the  limits of  weighted (with weight $\Phi((x, \psi(x))^{-\frac{2}{n+2}} $) ``volume" differences of the log concave function $f=e^{-\psi}$  and its floating function $f_\delta^\Phi$ or the convex function $\psi$ and its floating function $\psi_\delta^\Phi$  (with weight $\Phi((x, \psi(x))^{-\frac{2}{n+2}}  f(x) $). This resembles the behavior of the classical affine surface area for convex bodies of \cite{SchuettWerner1990}, mentioned in (\ref{eq:5-2-24B}). 
\par
 
Another definition of affine surface area for log-concave $f$ or convex functions $\psi$, respectively, was given in \cite{CFGLSW}. For $\lambda \in \mathbb{R}$, it has been defined as
\begin{equation*}\label{CFGLSW}
\operatorname{as}_\lambda (f) = \operatorname{as}_\lambda (\psi) = \int_{} e^{(2\lambda-1)\psi(x)-\lambda \langle x, \nabla\psi(x)\rangle}\left(\det \, \nabla^2 \psi (x)\right)^\lambda dx.
\end{equation*}
If we choose $\Phi$ in \eqref{asf} accordingly, we recover $\operatorname{as}_\lambda$. Thus, $\operatorname{as}_\Phi$ is a generalization of the previously introduced notion.
 
Now, we give  reasons why we call  \eqref{asf} an \textit{affine surface area}. We first recall from \cite{Hug, Lutwak:1996, SchuettWerner2004} the definition  of  the $L_p$-affine surface areas $\operatorname{as}_p(K)$ for convex bodies $K\subset\R^n$.
For $- \infty \leq p \leq \infty$, $p \neq -n$, they are  defined as 
\begin{equation}\label{asp-K}
\operatorname{as}_p(K) = \int_{\partial K} \frac{ \kappa_K(z)^\frac{p}{n+p}}{\langle z, N_K(z) \rangle^\frac{n(p-1)}{n+p}} d\mu_K(z).
\end{equation}
In particular, if we choose $p=1$ we get back the (usual) affine surface area of $K$, that is, 
\begin{equation}\label{asK}
\operatorname{as}(K) = \operatorname{as}_1(K) =  \int _{\partial K } \kappa_{K}(z)^\frac{1}{n+1} \, d{\mu_K}(z).
\end{equation}
 
Next, let us pass from integration over $\mathbb{R}^n$ in \eqref{asf} to integration over $\partial \operatorname{epi}(\psi)$
with the change of variable formula $\left(1 + \|\nabla \psi (x)\|^2\right)^\frac{1}{2}dx = d \mu_{\operatorname{epi}(\psi)}$. With \eqref{curvature} we get
\begin{equation*}
 \operatorname{as} _\Phi(f) =\int_{\partial \operatorname{epi}(\psi)} \left(\kappa_{\partial \operatorname{epi}(\psi)} (z)\right)^\frac{1}{n+2} 
\Phi(z)^{-\frac{2}{n+2}} e^{-\langle z,e_{n+1}\rangle}d \mu_{\operatorname{epi}(\psi)}(z)
\end{equation*}
and for $\Phi=\Phi_e$,
\begin{equation}
 \operatorname{as} _{\Phi_e}(f) =\int_{\partial \operatorname{epi}(\psi)} \left(\kappa_{\partial \operatorname{epi}(\psi)} (z)\right)^\frac{1}{n+2}  e^{-\frac{n}{n+2} z_{n+1}}d \mu_{\operatorname{epi}(\psi)}(z).
\end{equation}
Thus the expression  \eqref{asf}  coincides for the  unbounded convex set $\operatorname{epi}(\psi)$  with the one  for the affine surface area of a convex body in $\mathbb{R}^{n+1}$, as given in \eqref{asK}, modulo an additional weight function. This is one reason  to call the quantity the affine surface area of $f$.  

Another reason is the observation that $ \operatorname{as}_\Phi(f)$  shares several properties with $ \operatorname{as}_p(K)$. Firstly,  an  affine invariance property holds with the same degree of homogeneity as for the affine surface area for convex bodies in $\mathbb{R}^{n+1}$. Formally, for all   affine transformations $A: \R^n \rightarrow \R^n$ such that $\det A \neq 0$, one has that
$$
\operatorname{as}_\Phi(f \circ A) = |\det A |^{-\frac{n}{n+2}} \   \operatorname{as}_\Phi(f).
$$
This identity is  easily checked using  the fact that $\Hess_x (\psi\circ A)=A^T\Hess_{Ax}\psi A$. Moreover, as for convex bodies, a valuation property holds for $\operatorname{as}_\Phi(f)$. Namely, for log concave functions $f_1=e^{-\psi_1}$ and  $f_2=e^{-\psi_2}$ one has that
that 
\[ \operatorname{as}_\Phi(f_1)+ \operatorname{as}_\Phi(f_2) =  \operatorname{as}_\Phi(\max (f_1, f_2)) + \operatorname{as}_\Phi(\min (f_1, f_2)), 
\]
provided the function $\min(\psi_1,\psi_2)$ is convex as well. 
 
Yet another reason comes from  the next  observation, which shows that the definition for affine surface area for a function agrees with the definition for convex bodies if we choose as a function the gauge function $\| \,\cdot\, \|_K$ of a  convex body $K$ containing the origin in its interior. More explicitly, 
$$
\|x\|_K = \min\{ \alpha\geq 0: \ x \in  \alpha K\} = \max_{y \in K^\circ} \langle x, y \rangle = h_{K^\circ} (x),\qquad x\in\R^n,
$$ 
where $h_{K^\circ}$ stands for the support function of the polar body $K^\circ$ of $K$.
If we choose $\psi(x) = \frac{\|x\|_K^2}{2}$, then 
\begin{equation}\label{as-Gleichung}
  \operatorname{as}_{\Phi_e} \Big(\frac{\|\cdot\|_K^2}{2}\Big) = \Big(1+\frac{2}{n}\Big)^\frac{n}{2} \, \frac{(2\pi)^\frac{n}{2}}{n \,  \vol_n (B_2^n)}\  \operatorname{as}_\frac{n}{n+1}(K).
\end{equation}
Although a similar observation was already made  in \cite{CFGLSW}, we include the argument leading to \eqref{as-Gleichung} for completeness. We integrate in spherical coordinates with respect to the normalized cone measure  $\sigma_K$  associated with $K$. 
Thus, if we write $x=r\theta$, with $\theta\in\partial K$, then $dx=n \ \vol_n (K) r^{n-1}drd\sigma_K(\theta)$ and we get
\begin{eqnarray*}
 \operatorname{as}_{\Phi_e}  \left(\frac{\|\cdot\|_K^2}{2}\right)&=&  \int_{\mathbb R^{n}} \det\left( \nabla^2( \psi (x)) \right)^\frac{1}{n+2} \  e^{-\psi(x)} dx\\
 &=&n \ \vol_n (K) \int_0^{+\infty}r^{n-1}e^\frac{-n r^2}{2(n+2)}dr \int_{\partial K}  \left(\det \,\nabla^2  \psi (\theta) \right)^\frac{1}{n+2} \ d\sigma_K(\theta)\\
 &=& \left(1+\frac{2}{n}\right)^\frac{n}{2} \, (2\pi)^\frac{n}{2}\frac{\vol_n (K)}{\vol_n (B_2^n) }\int_{\partial K}   \left(\det \, \nabla^2 \psi (\theta) \right)^\frac{1}{n+2} \ d\sigma_K(\theta).
\end{eqnarray*}
The normalized cone measure $\sigma_K$ is related to the surface measure $\mu_K$ on $\partial K$ by
$$
d\sigma_K(x)=\frac{\langle \theta,N_K(\theta)\rangle d\mu_K(\theta)}{n \   \vol_n (K)}.
$$
Moreover, Lemma 1 of \cite{CFGLSW} and its proof in  show that
$
\det \, \nabla^2 \psi (\theta) = \frac{\kappa_K(\theta)}{\langle \theta,N_K(\theta)\rangle ^{n+1}}
$.
Thus,
\begin{align*}
&\operatorname{as}_{\Phi_e} \left(\frac{\|\cdot\|_K^2}{2}\right)\\
&= \left(1+\frac{2}{n}\right)^\frac{n}{2} \frac{(2\pi)^\frac{n}{2}}{n \vol_n (B_2^n)}\int_{\partial K}  \left(\frac{ \kappa_K (x) }{ \langle x,N_K(x)\rangle^{n+1} }\right)^\frac{1}{n+2} \langle x,N_K(x)\rangle\, d\mu_K(x)\\
&=\left(1+\frac{2}{n}\right)^\frac{n}{2} \, \frac{(2\pi)^\frac{n}{2}}{n  \vol_n (B_2^n)}\  \operatorname{as}_\frac{n}{n+1}(K).
\end{align*}
\par
 Finally, the most compelling  reason to call the quantity  $\operatorname{as}_\Phi(f)=\operatorname{as}_\Phi(\psi)$ a weighted  affine surface area
is the following theorem proved in \cite{Werner2002} in the case of convex bodies in $\mathbb{R}^{n}$. It says that with the weighted floating body $K^\Phi_\delta$, 
$$
\lim_{\delta \rightarrow 0} \frac{\vol_n(K) - \vol_n(K^\Phi_\delta)}{\delta^{2/(n+1)}} = c_n \  \int _{\partial K } \left( \kappa_{K}(z)\right)^\frac{1}{n+1} \Phi(z) ^{-\frac{2}{n+1}} d_{\mu_K}(z),
$$
where $c_n=\frac{1}{2} \left(\frac{n+1}{\vol_{n-1} (B^{n-1}_2)}\right)^\frac{2}{n+1}$, see also \eqref{eq:5-2-24B} in the introduction, 
and \cite{HuangSlomkaWerner, SchuettWerner2003, SchuettWerner2004} for a weighted random  analogue to (\ref{random}).
Theorem \ref{theo:f-deltafloat1}  and Theorem \ref{theo:f-deltafloat2} are their analogues for log concave functions. Thus these theorems provide  a  geometric description   of  
weighted affine  surface area for such functions.

\subsection{The case of $s$-concave functions}

In this section we present our main result for $s$-concave functions. 
Throughout this section we denote
$$
c_{n,s} = {s\over 2}\Bigl( \frac{n+s+1}{(n+s)\vol_{n+s}(B_2^{n+s})}\Bigr)^{\frac{2}{n+s+1}}.
$$

\begin{theorem}\label{thm:sconcave}
Let $n,s\in\mathbb{N}$ and $f:\R^n\to\R_+$ an $s$-concave function. Further, let $\Phi:\R^{n+s}\to\R_+$ be a continuous function satisfying $\Phi(x,y)=\phi(x,\|y\|)$ for some $\phi:\R^{n+1}\to\R_+$. Then,
\begin{align*}
&\lim_{\delta\to 0} \frac{\int_{\R^n} f(x)-f^\Phi_\delta(x)\,dx}{\delta^{2/(n+s+1)}} \\
&\qquad= c_{n,s}\,\int_{\R^n}|\det \nabla^2f(x)^{1/s}|^{1\over n+s+1}f(x)^{\frac{(s-1)(n+s)}{s(n+s+1)}}\phi(x,f^{1/s}(x))^{-\frac{2}{n+s+1}}\, dx.
\end{align*}
\end{theorem}

Let us consider the unweighted case in which $\phi\equiv 1$, or equivalently $\psi\equiv 1$. In this situation let us write $f_\delta$ instead of $f^\Phi_\delta$. Then Theorem \ref{thm:sconcave} implies that
	\begin{align}\label{s-ASA}
	&\lim_{\delta\to 0} \frac{\int_{\R^n} f(x)-f_{\delta}(x)\, dx}{\delta^{2/(n+s+1)}} \nonumber \\
	&\qquad={s\over 2}\Bigl( \frac{n+s+1}{(n+s)\vol_{n+s}(B_2^{n+s})}\Bigr)^{\frac{2}{n+s+1}}\int_{\R^n}|\det \nabla^2f(x)^{1/s}|^{{1\over n+s+1}}f(x)^{\frac{(s-1)(n+s)}{s(n+s+1)}}\, dx.
\end{align}

If $f$ is $s$-concave we can associate with it the convex function
\begin{equation}\label{eq:psif}
\psi_f(x):=s(1-f^{1/s}(x)),\qquad x\in\R^n.
\end{equation}
For $\lambda\in\R$ the $\lambda$-affine surface area of the $s$-concave function $f$ has been introduced in \cite[Definition 2]{CFGLSW} as
\begin{equation}\label{eq:5-2-24A}
\operatorname{as}_\lambda(f) = {1\over 1+ns}\int_{\R^n}{(1-{1\over s}\psi_f(x))^{(s-1)(1-\lambda)}(\det\nabla^2\psi_f(x))^\lambda\over (1+{1\over s}(\langle x,\nabla\psi_f(x)\rangle-\psi_f(x)))^{\lambda(n+s+1)-1}}\,dx,
\end{equation}
noting that the parameter $s$ in \cite{CFGLSW} plays the role of $1/s$ in our set-up. In particular, choosing $\lambda={1\over n+2}$ we can take the limit in \eqref{eq:5-2-24A} as $s\to\infty$ to see that
\begin{align*}
	\lim_{s\to\infty}(1+ns)\operatorname{as}_{1\over n+2}(f) &= \int_{\R^n} {e^{-\psi_f(x)(1-{1\over n+2})}(\det\nabla^2\psi_f(x))^{1\over n+2}\over e^{\lambda(\langle x,\nabla\psi_f(x)\rangle-\psi_f(x))}}\,dx\\
	&= \int_{\R^n} {e^{-(\psi_f(x)+{1\over n+2}\langle x,\nabla\psi_f(x)\rangle)}}(\det\nabla^2\psi_f(x))^{1\over n+2}\,dx\\
	&=\int_{\R^n}(\det\nabla^2\psi_f(x))^{1\over n+2}\Phi_f(x,\psi_f(x))^{-{2\over n+2}}e^{-\psi_f(x)}\,dx\\
	&=\operatorname{as}_{\Phi_f}(\psi_f)
\end{align*}
with the weight function
\begin{align*}
	\Phi_f(x,\psi_f(x)) = e^{2\langle x,\psi_f(x)\rangle}. 
\end{align*}
In other words, in the limit we arrive at a particular weighted surface area of the convex function $\psi_f$ associated with $f$ as introduced in Section \ref{sec:ConvexLogConcaveResults}.

Moreover, using \eqref{eq:psif} and the fact that $\nabla\psi_f=-s\nabla f^{1/s}$ and $\det\nabla^2\psi_f(x)=\det(-s\nabla^2 f^{1/s})=s^n\det(-\nabla^2 f^{1/s})$, we see that
\begin{equation*}
\operatorname{as}_\lambda(f) = \frac{s^{n \lambda  +1}}  {n+s} \int_{\mathbb{R}^n} \frac{f^\frac{1}{s}(x)^{ (s-1)(1-\lambda)}
	(\det (-\Hess (f^\frac{1}{s} (x) )))^\lambda}
{(f^\frac{1}{s}(x) -\langle x , \nabla (f^\frac{1}{s})(x) \rangle)^{\lambda(n+s+1) - 1 }} \  dx.
\end{equation*}
In particular, choosing $\lambda=\lambda_{n,s}={1\over n+s+1}$, this reduces to
\begin{align*}
\operatorname{as}_{\lambda_{n,s}}(f) &= \frac{s^{2n+s+1\over n+s+1}}  {n+s} \int_{\mathbb{R}^n} f^\frac{1}{s}(x)^{ (s-1)(n+s)\over n+s+1}
(\det (-\Hess (f^\frac{1}{s} (x) )))^{1\over n+s+1}\  dx\\
&= \frac{s^{2n+s+1\over n+s+1}}  {n+s} \int_{\mathbb{R}^n} f^\frac{1}{s}(x)^{ (s-1)(n+s)\over n+s+1}
|\det \Hess f^\frac{1}{s} (x) |^{1\over n+s+1}\  dx
\end{align*}
which in turn coincides with \eqref{s-ASA}.
Moreover, according to \cite[Equation (31)]{CFGLSW} we have that
$$
\operatorname{as}_\lambda(f) = {s^{n/2+1}\over(n+s)\vol_{s-1}(S^{s-1})}\operatorname{as}_p(K_f^s),
$$
where on the right hand side we have the $L^p$-affine surface area of the auxiliary convex body $K_f^s$ with $p=(n+s){\lambda\over1-\lambda}$. Using this identity with $\lambda=\lambda_{n,s}$ and noting that in this case $p=1$, we conclude that
$$
\operatorname{as}_{\lambda_{n,s}}(f) = {s^{n/2+1}\over(n+s)\vol_{s-1}(S^{s-1})}\operatorname{as}(K_f^s),
$$
with $\operatorname{as}(K_f^s)=\operatorname{as}_1(K_f^s)$ being the usual affine surface area of $K_f^s$ as defined by \eqref{asK} and $S^{s-1}$ denoting the $(s-1)$-dimensional unit sphere.

The previous discussion motivates the following definition.

\begin{definition}
Let $n,s\in\mathbb{N}$ and $f:\R^n\to\R_+$ an $s$-concave function. Further, let $\Phi:\R^{n+s}\to\R_+$ be a continuous function satisfying $\Phi(x,y)=\phi(x,\|y\|)$ for some $\phi:\R^{n+1}\to\R_+$. The expression
\begin{equation}
	\label{def:assphi}
\operatorname{as}_\Phi^s(f) = 
\int_{\R^n}|\det \nabla^2f(x)^{1/s}|^{1\over n+s+1}f(x)^{\frac{(s-1)(n+s)}{s(n+s+1)}}\phi(x,f^{1/s}(x))^{-\frac{2}{n+s+1}}\, dx
\end{equation}
is called the $\phi$-weighted affine surface area of  the $s$-concave function.
\end{definition}

 \section{Proofs}\label{proofs}
 
 \subsection {Proof of Theorems \ref{theo:f-deltafloat2}  and  \ref{theo:f-deltafloat1}, and Propositions \ref{corollary1} and \ref{corollary2}} 
 
We follow the proof of the main theorem of \cite{LiSchuettWerner}, where the results have been obtained for the unweighted case $\Phi \equiv 1$. As a preparation, we need several lemmas.  These  lemmas and their proofs are  the analogous ones of \cite{LiSchuettWerner}. The first lemma is  well known, see, e.g., \cite{SchuettWerner2003}.  

 \begin{lemma}\label{lemma:cap}
 Let $a_1,\ldots,a_n>0$,
$$
\mathcal{E}=\left\{x\in\mathbb{R}^n:
\ \sum_{i=1}^{n}\left|\frac{x_{i}}{a_{i}}\right|^{2}\leq
1 \right\}.
$$ 
and let $H_h=H((a_{n}-h)e_{n},e_{n})$. Then for all
$h\leq a_{n}$,
\begin{eqnarray*}
h^\frac{n+1}{2} \ \left( 1- \frac{h}{2a_n}\right)^\frac{n-1}{2} \  
\leq \  
\frac{ (n+1) \  a_n^\frac{n-1}{2}\   \vol_{n}(\mathcal{E}\cap H_h^{-})}{2^\frac{n+1}{2}
\vol_{n-1}(B_{2}^{n-1})\  \prod_{i=1}^{n-1}a_{i}}\  
\leq\  
h^\frac{n+1}{2}.
\end{eqnarray*}
In particular, if $\mathcal{E}= r B^n_2$ is a Euclidean ball  with radius  $r$  in $\mathbb{R}^n$, then for all $u \in S^{n-1}$, for $h \leq r$ and $H_h=H((r-h)u, u )$, 
\begin{eqnarray*}
 h^\frac{n+1}{2}  \left( 1-\frac{ h}{2r} \right)  \  \leq \   \frac{(n+1)\   \vol_{n}\left(rB^n_2\cap H_h^{-}\right)}{2^\frac{n+1}{2} \ \vol_{n-1}\left(B^{n-1}_2\right) \  r^\frac{n-1}{2}}
  \leq 
  \ h^\frac{n+1}{2}. 
\end{eqnarray*}
 \end{lemma}

 \begin{lemma}\label{lemma:f-delta}
 Let $\psi \in \text{Con}(\mathbb{R}^n)$ and $\Phi$ be a continuous, strictly positive function on $\mathbb{R}^{n+1}$. 
\begin{itemize}
\item[(i)] Let $x \in \mathbb R^{n}$ be such that the Hessian $\nabla^2 \psi$ at $x$  is positive definite. Then there are constants $\beta_1$ and $\beta_2$ such that  for all $\varepsilon >0$ there is $\delta_0= \delta_0(x, \varepsilon)$ such that for all $\delta \leq \delta_0$, 
\begin{eqnarray*}
	&&\hskip -5mm  (1- \beta_2 \varepsilon) \ 
	c_{n+1} \left(\det\left( \nabla^2 \psi (x)\right) \right)^\frac{1}{n+2} \Phi((x, \psi(x))^{-\frac{2}{n+2}} \, \leq \,  \frac{\psi^\Phi_\delta (x) - \psi(x) }{\delta^{2/(n+2)} }  \\
	&& \hskip 25mm \leq(1+ \beta_1 \varepsilon) \ 
	c_{n+1}  \left(\det\left( \nabla^2 \psi (x)\right) \right)^\frac{1}{n+2} \Phi((x, \psi(x))^{-\frac{2}{n+2}}.
\end{eqnarray*}
Consequently, for $f = e^{-\psi}$ we get with (new)  constants $\beta_1$ and $\beta_2$
\begin{eqnarray*}
	&&\hskip -5mm (1- \beta_2 \varepsilon)   f(x) \  c_{n+1}  \left(\det\left( \nabla^2\psi (x)\right) \right)^\frac{1}{n+2} \Phi((x, \psi(x))^{-\frac{2}{n+2}} \leq  \frac{ f(x) - f^\Phi_\delta (x) }{\delta^{2/(n+2)} }  \\ 
	&& \hskip 25mm  \leq(1+ \beta_1 \varepsilon)   f(x) \  c_{n+1}  \left(\det\left( \nabla^2\psi(x) \right) \right)^\frac{1}{n+2} \Phi((x, \psi(x))^{-\frac{2}{n+2}}.
\end{eqnarray*}
In particular, when $\Phi=\Phi_e$, we get that for all $\varepsilon >0$ there is $\delta_0= \delta_0(x, \varepsilon)$ such that for all $\delta \leq \delta_0$, 
\begin{eqnarray*}
	&&\hskip -5mm  (1- \beta_2 \varepsilon) \ 
	c_{n+1} \left(\det\left( \nabla^2 \psi (x)\right) \right)^\frac{1}{n+2} \, e^{\frac{2}{n+2} \psi(x) } \, \leq \,  \frac{\psi^e_\delta (x) - \psi(x) }{\delta^{2/(n+2)} }  \\
	&& \hskip 25mm \leq(1+ \beta_1 \varepsilon) \ 
	c_{n+1}  \left(\det\left( \nabla^2 \psi (x)\right) \right)^\frac{1}{n+2} \, e^{\frac{2}{n+2} \psi(x) },
\end{eqnarray*}
and for $f = e^{-\psi}$ we get with (new)  constants $\beta_1$ and $\beta_2$
\begin{eqnarray*}
	&&\hskip -5mm (1- \beta_2 \varepsilon)   f(x) \  c_{n+1}  \left(\det\left( \nabla^2\psi (x)\right) \right)^\frac{1}{n+2} \, e^{\frac{2}{n+2} \psi(x) } \leq  \frac{ f(x) - f^e_\delta (x) }{\delta^{2/(n+2)} }   \\ 
	&& \hskip 25mm  \leq(1+ \beta_1 \varepsilon)   f(x) \  c_{n+1}  \left(\det\left( \nabla^2\psi(x) \right) \right)^\frac{1}{n+2} \, e^{\frac{2}{n+2} \psi(x) }.
\end{eqnarray*}

\item[(ii)] Let $x \in \mathbb R^{n}$ be such that $ \det\left(\nabla^2 \psi(x)\right) =0$. Then for all $\varepsilon >0$ there is $\delta_0= \delta_0(x, \varepsilon)$ such that for all $\delta \leq \delta_0$, 
\begin{eqnarray*}
	0 \leq \frac{\psi^\Phi_\delta (x) - \psi(x) }{\delta^{2/(n+2)} }  \leq \varepsilon.
\end{eqnarray*}
Consequently,  for $f = e^{-\psi}$ we get for all $\varepsilon >0$ that there is $\delta_0= \delta_0(x, \varepsilon)$ such that for all $\delta \leq \delta_0$, 
\begin{eqnarray*}
	0 \leq \frac{f(x) - f^\Phi_\delta (x) }{\delta^{2/(n+2)} }  \leq \varepsilon.
\end{eqnarray*}
In particular, when $\Phi=\Phi_e$, we get that 
\begin{eqnarray*}
	0 \leq \frac{\psi^e_\delta (x) - \psi(x) }{\delta^{2/(n+2)} }  \leq \varepsilon  \hskip 5mm \text{and} \hskip 5mm 0 \leq \frac{f(x) - f^e_\delta (x) }{\delta^{2/(n+2)} }  \leq \varepsilon.
\end{eqnarray*}
\end{itemize}
 \end{lemma}
\begin{proof}
Let $\varepsilon >0$ be given and let 
$x_0 \in \mathbb R^{n}$.  We put $z_{x_0}=(x_0, \psi(x_0))$ and  denote by $N_\psi(z_{x_0})$ the outer unit normal in $z_{x_0}$  to the surface described by $\psi$.  As recalled above, $N_\psi(z_{x_0})$ exists uniquely for almost all $x_0$.

To prove part (i) of the lemma, let $x_0$ be such that the Hessian $\nabla^2 \psi (x_0)$  is positive definite. Then, 
locally around $z_{x_0}$, the graph of $\psi$ can be approximated by an ellipsoid $\mathcal{E}$. We  make this precise. Let $\mathcal{E}$ be such that the lengths of its principal axes are $a_1, \dots, a_{n+1}$ and such that its center 
is at $z_{x_0} - a_{n +1} N_\psi(z_{x_0})$.  Let $\mathcal{E} (\varepsilon^-)$ be the ellipsoid centered at $z_{x_0}-  a_{n+1} N_\psi(z_{x_0})$ 
whose principal axes  coincide with the ones of $\mathcal{E}$,  but have lengths $(1-\varepsilon) a_1, \dots, (1-\varepsilon) a_{n}, a_{n+1}$. Similarly, let $\mathcal{E} (\varepsilon^+)$ be the ellipsoid centered at $z_{x_0}-  a_{n+1} N_\psi(z_{x_0})$, with the same principal axes as $\mathcal{E}$,  but with lengths $(1+\varepsilon) a_1, \dots, (1+\varepsilon) a_{n}, a_{n+1}$.
Then,
$$ z_{x_0} \in \partial \mathcal{E}  \hskip 3mm \text { and } \hskip 3mm N_{\mathcal{E}}(z_{x_0}) = N_\psi(z_{x_0}),$$
and (see, e.g., \cite{SchuettWerner2003})  there exists a $\Delta_\varepsilon >0$ 
such that 
\begin{eqnarray}\label{ellipse}
&& \hskip -10mm H^-\left( z_{x_0} - \Delta_\varepsilon N_\psi(z_{x_0}), N_\psi(z_{x_0})\right)  \  \cap  \  \mathcal{E} (\varepsilon^-) \nonumber \\
&&  \subseteq  H^-\left( z_{x_0} - \Delta_\varepsilon N_\psi(z_{x_0}), N_\psi(z_{x_0})\right) \  \cap \   \{(x,y):   y \geq \psi(x)\} \nonumber \\
&& \hskip 10mm \subseteq H^-\left( z_{x_0} - \Delta_\varepsilon N_\psi(z_{x_0}), N_\psi(z_{x_0})\right) \  \cap  \  \mathcal{E} (\varepsilon^+) .
\end{eqnarray}
For  $\delta \geq 0$, let $z_\delta^\Phi= (x_0, \psi_\delta^\Phi(x_0))$. We choose $\delta$ so small that  for all support hyperplanes $H(z_\delta^\Phi)$ to $\left(\operatorname{epi}(\psi)\right)_\delta^\Phi$  through $z_\delta^\Phi$ we have
$$
  H(z_\delta^\Phi) ^{-}  \cap  \mathcal{E} (\varepsilon^-) \subseteq   H^-\left( z_{x_0} - \Delta_\varepsilon N_\psi(z_{x_0}), N_\psi(z_{x_0})\right)  \  \cap  \  \mathcal{E} (\varepsilon^-) .
$$
Let $\Delta_{\delta}$ be such that
\begin{equation}\label{DefDeltadelta}
H(z_{x_{0}}-\Delta_{\delta}N_{\psi}(z_{x_{0}}),N_{\psi}(z_{x_{0}}))
\end{equation}
is a supporting hyperplane to $\operatorname{epi}(\psi_{\delta}^\Phi)$.
We choose $\delta$ so small that $\Delta_\delta \leq \Delta_\varepsilon$ of  \eqref{ellipse}.
As $\Phi$ is continuous on $\mathbb{R}^{n+1}$, there is a neighborhood $U$ of $z_{x_0}$ such that for all $z \in U$,
\begin{equation}\label{stetig}
(1-\varepsilon) \Phi(z_{x_0}) \leq \Phi(z) \leq (1+\varepsilon) \Phi(z_{x_0}).
\end{equation}
We choose $\delta$ so small that  for all support hyperplanes $H(z_\delta^\Phi)$ to $\left(\operatorname{epi}(\psi)\right)_\delta^\Phi$  through $z_\delta^\Phi$ we have
$$
  H(z_\delta^\Phi) ^{-}  \cap  \mathcal{E} (\varepsilon^-) \subseteq   U.
$$
 As $\partial \operatorname{epi}(\psi)$ is approximated by the boundary of an ellipsoid in $z_{x_0}$, we have that $z_\delta^\Phi \in \text{int}(\operatorname{epi}(\psi))$.  
 Thus we get 
by definition of $\left(\operatorname{epi}(\psi)\right)_\delta^\Phi$, respectively $\psi_\delta^\Phi$, by Proposition \ref{properties} and Lemma \ref{lemma:cap},
\begin{eqnarray*}
\delta &\leq&  \int_{H\left( z_{x_0} - \Delta_\delta N_\psi(z_{x_0}), N_\psi(z_{x_0} )\right)^- \, \cap \,\operatorname{epi}(\psi)} \Phi(z) dz \\
&\leq&
(1+\varepsilon) \Phi(z_{x_0})\,  \vol_{n+1} \left(H\left( z_{x_0} - \Delta_\delta N_\psi(z_{x_0}), N_\psi(z_{x_0})\right)^- \,   \cap \, \operatorname{epi}(\psi)   \right) \\
&\leq & (1+\varepsilon) \Phi(z_{x_0}) \vol_{n+1} \left(H\left( z_{x_0} - \Delta_\delta N_\psi(z_{x_0}), N_\psi(z_{x_0})\right) ^-  \cap \   \mathcal{E} (\varepsilon^+) \right)  \\ 
&\leq& (1+\varepsilon)^{n+1} \,  \Phi(z_{x_0})   \frac{2^\frac{n+2}{2} \vol_n(B^n_2)} {n+2}  \   \prod_{i=1}^n \frac{a_i}{\sqrt{a_{n+1}}}  \ \Delta_\delta ^\frac{n+2}{2}.
\end{eqnarray*} 
As
$\kappa_{\psi} (z_{x_0}) =  \prod_{i=1}^n \frac{a_{n+1}}{a_i^2}$ (see, e.g., \cite{SchuettWerner2003}),  \eqref{curvature} yields
\begin{eqnarray}\label{below1}
\Delta_\delta \geq  \frac{c_{n+1} }{(1+ \varepsilon)^{2 \frac{n+1}{n+2}}} \   \frac{\left(\det \nabla^2 \psi(x_0) \right)^\frac{1}{n+2}}{\left(1 + \|\nabla \psi(x_0)\|^2\right)^\frac{1}{2} } \,  
\frac{\delta^{2/(n+2)}}{\Phi(z_{x_0})^\frac{2}{n+2}} , 
\end{eqnarray}
where $c_{n+1}$ is as given by \eqref{cn+1}.  By \eqref{DefDeltadelta}
$$
\Delta_\delta \leq  \langle N_{\psi} (z), e_{n+1}\rangle \left( \psi^\Phi_\delta(x_0)  - \psi(x_0)\right) .
$$
Therefore, with \eqref{normal},
$$
\Delta_\delta \leq  \frac{\psi^\Phi_\delta(x_0)  - \psi(x_0)}{(1+\|\nabla \psi(x_0)\|^2)^{\frac{1}{2}}} 
$$
and thus with \eqref{below1},
\begin{equation}\label{below2} 
\psi^\Phi_\delta(x_0)  - \psi(x_0)  \geq \frac{c_{n+1} }{(1+ \varepsilon)^{2 \frac{n+1}{n+2}}} \   \left(\det \nabla^2 \psi (x_0)\right)^\frac{1}{n+2}\, \frac{ \delta^{2/(n+2)}}
{\Phi(z_{x_0})^\frac{2}{n+2}}.
\end{equation}

Now we estimate $\delta$ from below. 
As $\nabla^2 \psi (x_0)$  is positive definite, we have that for sufficiently small $\delta$,  $z_\delta^\Phi \in  {\rm int} (\operatorname{epi}(\psi))$. 
By Proposition \ref{properties}, \eqref{stetig} and \eqref{ellipse} there exists a hyperplane $H_\delta$ such that 
\begin{eqnarray} \label{unten}
\delta &=& \int_{H_\delta^- \cap  \operatorname{epi}(\psi)} \Phi(z) dz  \geq (1-\varepsilon) \Phi(z_{x_0}) 
\vol_{n+1}\left( H_\delta^- \  \cap \   \operatorname{epi}(\psi)  \right) \nonumber\\
&\geq& (1-\varepsilon) \Phi(z_{x_0})\vol_{n+1}\left( H_\delta^- \  \cap \   \mathcal{E} (\varepsilon^-) \right).
\end{eqnarray}
The expression $\vol_{n+1}\left( H_\delta^- \  \cap \   \mathcal{E} (\varepsilon^-) \right)$ is invariant under affine transformations with determinant $1$. We apply an affine transformation that 
maps $\mathcal {E} (\varepsilon^-)$ into a Euclidean ball with radius 
\begin{equation} \label{radius}
r= (1- \varepsilon) \  \left(\frac{1} {\kappa_{\psi} (z_{x_0})} \right)^ \frac{1}{n}.
\end{equation}
Now we use   Lemma 11 of \cite{SchuettWerner1990} and note that $z_{x_0}$ corresponds to $0$ of Lemma 11, that $z_\delta$ corresponds to $z$ and that $N_\psi(z_{x_0})$
corresponds to $N(0)=(0, \ldots, 0,-1)$. We choose $\varepsilon <\varepsilon_0$, where  $\varepsilon_0$ is given by Lemma 11,  and we choose $\delta$ so small that 
$\psi^\Phi_\delta(x_0) - \psi(x_0) = \|z^\phi_\delta - z(x_0)\| \leq \varepsilon <\varepsilon_0$.
By  Lemma 11 (iii)  of \cite{SchuettWerner1990},
\begin{eqnarray*}
\vol_{n+1}\left( H_\delta^- \  \cap \   \mathcal{E} (\varepsilon^-) \right) &=& \vol_{n+1}\left( H_\delta^- \  \cap \  B^{n+1}_2\left( z_{x_0}- r\  N_\psi(z_{x_0}),  r\  \right)\right) \\
& \geq& \eta(\gamma)^{-n} \  \vol_{n+1}\left(C( r, d_0 (1- c(\eta(\gamma)-1)))\right),
\end{eqnarray*}
where $c$ is an absolute constant, $C( r, d_0 (1- c(\eta(\gamma)-1)))$ is the cap of the $(n+1)$-dimensional Euclidean ball $B^{n+1}_2\left( z_{x_0}- r\  N_\psi(z_{x_0}),  r\  \right)$ of height $d_0 (1- c(\eta(\gamma)-1)))$
and $d_0$ is the  distance from $z_\delta$ to the boundary of $B^{n+1}_2\left( z_{x_0}- r\  N_\psi(z_{x_0}),  r\  \right)$. 
$\gamma = 4\sqrt{2rd_0}$ and $\eta$ is  a monotone function on $\mathbb{R}^+ $ such that $\lim_{t \rightarrow 0} \eta(t) =1$.
Thus, by \eqref{unten} and Lemma \ref{lemma:cap}, 
\begin{align}\label{delta1}
\delta 
&\geq (1-\varepsilon) \Phi(z_{x_0})\,   \frac{2^\frac{n+2}{2} \vol_{n}(B^n_2)} {(n+2)\  \eta(\gamma)^n}  r^\frac{n}{2}   \nonumber \\
&\qquad  \times  \left(d_0 (1- c(\eta(\gamma)-1)))\right)^\frac{n+2}{2}  
\left(1- \frac{d_0 (1- c(\eta(\gamma)-1)) }{2r}\right)^\frac{n}{2}.
\end{align}
We apply Lemma 11 (ii)  of \cite{SchuettWerner1990} next and note that $z_n$ of Lemma 11  corresponds to  
$z_n= \langle e_{n+1}, N_\psi(z_{x_0}) \rangle   \left(\psi^\Phi_\delta(x_0) - \psi(x_0) \right)$ in our case and $\frac{\xi }{\|\xi\| }= e_{n+1}$. 
Then by Lemma 11 (ii), 
\begin{equation}\label{d0}
d_0 \leq  \langle e_{n+1}, N_\psi(z_{x_0}) \rangle \  \left(\psi^\Phi_\delta(x_0) - \psi(x_0) \right) \leq d_0 + \frac{2 d_0^2}{r  \ \left| \langle e_{n+1}, N_{\psi}(z_{x_0})\rangle\right|^2}.
\end{equation}
Thus  we get  for sufficiently small $\delta$,  with an absolute constant $\beta_1$, that 
\begin{equation}\label{eta1}
\eta(\gamma) = \eta (4\sqrt{2rd_0})  \leq 1+\beta_1 \varepsilon
\end{equation}
and hence 
\begin{equation}\label{eta2}
1 - c (\eta(\gamma) -1) \geq 1 - \beta_2 \varepsilon, 
\end{equation}
with an  absolute constant $\beta_2$.
It follows from \eqref{d0} that 
$$
d_0 \geq \langle e_{n+1}, N_\psi(z_{x_0}) \rangle   \left(\psi^\Phi_\delta(x_0) - \psi(x_0) \right) 
\left( 
1-\frac{2\  \left(\psi^\Phi_\delta(x_0) - \psi(x_0)\right)}{r\  \left( \langle e_{n+1}, N_\psi(z_{x_0}) \rangle \right)} 
\right).
$$
We conclude with \eqref{normal}, \eqref{delta1},  \eqref{d0},  \eqref {eta1}  and \eqref {eta2}  that with (new) absolute constants $\beta_1, \beta_2$, 
\begin{eqnarray*}
&&\delta^{2/(n+2)} 
\geq\\
&& \frac{ 1-\beta_2\   \varepsilon}{(1+\beta_1\   \varepsilon)^\frac{2n}{n+2}}\, \Phi(z_{x_0})^\frac{2}{n+2}  \,  \frac{r^\frac{n}{n+2}}{c_{n+1}}\  \frac{\left(\psi^\Phi_\delta(x_0) - \psi(x_0) \right)}{(1+\|\nabla \psi(x_0)\|^2)^{\frac{1}{2}}} \left(1- 2 \frac{\psi^\Phi_\delta(x_0) - \psi(x_0) }{ r  (1+\|\nabla \psi (x_0)\|^2)^{\frac{1}{2}}}\right) ^{2\frac{n+1}{n+2}}.
\end{eqnarray*}
For  $\delta$  small enough, 
\eqref{curvature} and \eqref{radius} give, with (new) absolute constants $\beta_1, \beta_2$,
\begin{eqnarray}\label{below3}
\psi^\Phi_\delta(x_0)  - \psi(x_0)  \leq \frac{(1+ \beta_1\ \varepsilon)^\frac{2n}{n+2} }{(1- \beta_2\ \varepsilon)^{2\frac{n+1}{n+2}}}  \ c_{n+1} \  \left(\det \nabla^2 \psi(x_0)\right)^\frac{1}{n+2}  \,  \frac{\delta^{2/(n+2)}}{\Phi(z_{x_0})^\frac{2}{n+2}}.
\end{eqnarray}
This completes the proof of part (i).

For (ii), we assume that $\det \left(\nabla^2 \psi(x_0)\right) =0$.  Suppose first that there is $\delta_0$ such that $z^\Phi_{\delta_0} \in \partial \operatorname{epi}(\psi)$. Then 
$z^\Phi_\delta   \in \partial \operatorname{epi}(\psi)$ for all $\delta \leq \delta_0$.  As $z^\Phi_\delta = (x_0, \psi^\Phi_\delta (x_0))$ and $z_{x_0} = (x_0, \psi(x_0))$, we thus get  that  $\psi^\Phi_\delta(x_0)  = \psi(x_0)$ for all $\delta \leq \delta_0$, 
and hence $\frac{\psi^\Phi_\delta(x_0)  - \psi(x_0)}{\delta^{2/(n+2)}}=0$.

Suppose next that for all $\delta >0$, $z^\Phi_\delta  \in \text{int}(\operatorname{epi}(\psi))$.  As $\det \left(\nabla^2 \psi(x_0)\right) =0$,  the indicatrix of Dupin at $z_{x_0}$ is an elliptic cylinder 
and we may assume that the first $k$ axes have infinite lengths and the others do not. Then, 
(see e.g., the proof of Lemma 23 in \cite{SchuettWerner2004}),  for all $\varepsilon>0 $ there is an ellipsoid $\mathcal{E}$ and $\Delta_\varepsilon >0$ such that for all $\Delta \leq \Delta_\varepsilon$,
\begin{eqnarray} \label{ellzyl}
\mathcal{E}  \cap H^-(z_{x_0} - \Delta N_\psi(z_{x_0}), N_\psi(z_{x_0}))  \subset \operatorname{epi}(\psi) \cap H^-(z_{x_0} - \Delta N_\psi(z_{x_0}), N_\psi(z_{x_0}))
\end{eqnarray}
and such that the lengths of the $k$ first principal axes of $\mathcal{E}$ are larger than $\frac{1}{\varepsilon}$.  
By Proposition \ref{properties} and continuity of $\Phi$,  there exists a hyperplane $H_\delta$ such that 
\begin{eqnarray*} \label{unten}
\delta &=& \int_{H_\delta^- \cap  \operatorname{epi}(\psi)} \Phi(z) dz  \geq (1-\varepsilon) \Phi(z_{x_0}) 
\vol_{n+1}\left( H_\delta^- \  \cap \   \operatorname{epi}(\psi)  \right).
\end{eqnarray*}
We choose $\delta$ so small  that 
$$
\mathcal{E} \cap H_\delta^- \subset \mathcal{E} \cap H^-(z_{x_0} - \Delta N_\psi(z_{x_0}), N_\psi(z_{x_0})).
$$
Then we have with \eqref{ellzyl}, 
$$
\delta \geq  (1-\varepsilon) \Phi(z_{x_0}) \vol_{n+1}\left( H_\delta^- \  \cap \   \operatorname{epi}(\psi)  \right) \geq (1-\varepsilon) \Phi(z_{x_0})   \vol_{n+1}(\mathcal{E} \cap H_\delta^-) .
$$
Now we continue as in \eqref{unten} and thereafter, and conclude that
\begin{eqnarray*}
\frac{\psi^\Phi_\delta(x_0)  - \psi(x_0) }{\delta^{2/(n+2)}}  &\leq& \frac{(1+ \beta_1\ \varepsilon)^\frac{2n}{n+2} }{(1- \beta_2\ \varepsilon)^{2\frac{n+1}{n+2}}}  \, \frac{c_{n+1}}{\Phi(z_{x_0})^\frac{2}{n+2}} \,  \left(\prod_{i=1}^n \frac{a_i}{\sqrt{a_{n+1}}} \right)^{-\frac{2}{n+1}} \\
&\leq& \frac{(1+ \beta_1\ \varepsilon)^\frac{2n}{n+2} }{(1- \beta_2\ \varepsilon)^{2\frac{n+1}{n+2}}}  \, \frac{c_{n+1}}{\Phi(z_{x_0})^\frac{2}{n+2}} \  \left(\prod_{i=k+1}^n \frac{a_i}{\sqrt{a_{n+1}}} \right)^{-\frac{2}{n+1}} \  \varepsilon^{\frac{2k}{n+2}} ,
\end{eqnarray*}
where in the last inequality we have used 
 that for all $1 \leq i \leq k$, $a_i = \frac{1}{\varepsilon}$. This completes the proof.
\end{proof}

\vskip 4mm
 
We require  a  uniform bound in $\delta$ for the quantities $\frac{\psi^e_\delta(x)- \psi(x)}{\delta^{2/(n+2)}}$ and $\frac{\psi^\Phi_\delta(x)- \psi(x)}{\delta^{2/(n+2)}}$ so that we can eventually apply the dominated convergence theorem.
This is achieved in the next lemma.
 
 \begin{lemma}\label{lemma:interchange0}
 Let  $\psi \in \text{Con}(\mathbb{R}^n)$.  Let $\eta >0$ and 
let $\Phi: \mathbb{R}^{n+1} \to [\eta, \infty)$  be a continuous function.
There exists $\delta_0$ such that for all $\delta < \delta_0$,  for  all $x \in \mathbb R^{n}$, 
\begin{eqnarray*}
0 \leq \frac{\psi^\Phi_\delta(x)- \psi(x)}{\delta^{2/(n+2)}} \leq 2^\frac{3n+4}{n+2}
c_{n+1}  \frac{(1+\|\nabla \psi (x)\|^2)^{\frac{1}{2}}} {\eta^{\frac{2}{n+2}}\,  r_\psi(x) ^\frac{n}{n+2}},
\end{eqnarray*}
where $r_\psi(x)$ is as in \eqref{r(x)}. 
In particular, for  all $\delta < \delta_0$ and for all $x \in  \mathbb R^{n}$, we have that
\begin{eqnarray*}
0 \leq  \frac{ f(x) - f^\Phi_\delta(x) } {\delta^{2/(n+2)}} 
\leq  2^\frac{3n+4}{n+2}
c_{n+1}   \frac{(1+\|\nabla \psi (x)\|^2)^{\frac{1}{2}}} {\eta^{\frac{2}{n+2}}\,  r_\psi(x) ^\frac{n}{n+2}} \, f(x).
\end{eqnarray*}
\end{lemma}
\begin{proof}  
Let $z_x = (x, \psi(x)) \in \partial\left(\operatorname{epi}(\psi)\right)$ and let $z^\Phi_\delta =(x, \psi^\Phi_\delta(x))$.  
Let $r_\psi(x)$  be as in \eqref{r(x)}. If $N_{\psi} (z_x)$ is not unique, then $r_{\psi}(z_x)=0$ and the inequality holds trivially.
 Moreover, if $\psi^\Phi_\delta(x)= \psi(x)$, then $\frac{\psi^\phi_\delta(x)- \psi(x)}{\delta^{2/(n+2)}} =0$ and again, the inequality holds trivially.

Thus we can assume that  $N_{\psi} (z_x)$ is  unique  and $\psi^\phi_\delta(x) > \psi(x)$. By Proposition \ref{properties}, there is a hyperplane $H_\delta$ such that $z^\phi_\delta \in H_\delta$ and
\begin{eqnarray}\label{interchange1}
\delta &=& \int_{H_\delta^- \cap \operatorname{epi}(\psi)} \Phi(z) dz \geq \int_{H_\delta^- \cap B^{n+1}_2 \left( z_x-r_\psi(x) N_\psi(z_x), r_\psi(x) \right)} \Phi(z) dz \nonumber\\
 &\geq & \eta \,  \vol_{n+1}\left(H_\delta^- \cap B^{n+1}_2 \left( z_x-r_\psi(x) N_\psi(z_x), r_\psi(x) \right) \right).
\end{eqnarray}
The proof now continues as the of of Lemma 5 in  \cite{LiSchuettWerner}. We include it for completeness.
We will estimate $ \vol_{n+1}\left(H_\delta^- \cap B^{n+1}_2 \left( z_x-r_\psi(x) N_\psi(z_x), r_\psi(x) \right) \right)$. For this, we choose $\delta_0$ so small that  for all $\delta \leq \delta_0$,  $z_x \in H_\delta^-$.
 
We start by dealing with the case
\begin{equation}\label{assumption1}
\psi^\Phi_\delta(x) - \psi(x)  \geq  r_\psi(x) \   \langle e_{n+1}, N_\psi(z_x) \rangle. 
\end{equation}
In this case  we have for all hyperplanes $H(z^\Phi_\delta)$ through $z^\phi_\delta$  and such that $z_x \in H^-(z^\phi_\delta)$, 
\begin{eqnarray*} 
&& \hskip -20mm \vol_{n+1} \left(H^-(z^\Phi_\delta) \cap B^{n+1}_2 \left(z_x-r_\psi(x) N_\psi(z_x) , r_\psi(x) \right) \right) \geq  \\
&&\hskip 20mm \vol_{n+1} \left(H_0^-(z^\Phi_\delta) \cap B^{n+1}_2 \left(z_x-r_\psi(x) N_\psi(z_x) , r_\psi(x) \right) \right),
\end{eqnarray*} 
where $H_0(z^\Phi_\delta)$ is this hyperplane orthogonal to $x$ and  such that both, $z_x$ and $z^\phi_\delta$ are  in $H_0(z^\Phi_\delta)$.  We can estimate the latter from below by  the cone with base $\frac{1}{2} \left(\psi^\Phi_\delta(x) - \psi(x)\right)  B^n_2$ and height 
$h \geq \left(  \langle e_{n+1}, N_\psi(z_x) \rangle \right)^2 \  \frac{r_\psi(x)} {2}$. 
Hence, by \eqref{assumption1}, 
\begin{align*} 
&\vol_{n+1} \left(H_0^-(z^\Phi_\delta) \cap B^{n+1}_2 \left(z_x-r_\psi(x) N_\psi(z_x) , r_\psi(x) \right) \right)  \\
&\geq\frac{\vol_{n} \left(B^n_2\right) } {2^{n+1} (n+1) } \left(  \langle e_{n+1}, N_\psi(z_x) \rangle \right) ^2 \ r_\psi(x) \  \left(\psi^\Phi_\delta(x) - \psi(x) \right)^n\\
&\geq\frac{\vol_{n} \left(B^n_2\right) } {2^{n+1}  (n+1) } \left(  \langle e_{n+1}, N_\psi(z_x) \rangle \right) ^2 \ r_\psi(x) \ \left(\psi^\Phi_\delta(x) - \psi(x) \right)^\frac{n+2}{2} \left(\psi^\Phi_\delta(x) - \psi(x) \right)^\frac{n-2}{2}   \\
&\geq\frac{\vol_{n} \left(B^n_2\right) } {2^{n+1}  (n+1) } \left(  \langle e_{n+1}, N_\psi(z_x) \rangle \right) ^\frac{n+2}{2}  \ r_\psi(x)^\frac{n}{2} \ \left(\psi^\Phi_\delta(x) - \psi(x) \right)^\frac{n+2}{2}.
\end{align*}
Since  $ \langle e_{n+1}, N_{\psi}(z_{x_0})\rangle= \frac{1}{(1+\|\nabla \psi\|^2)^{\frac{1}{2}}}$,  we get with \eqref{interchange1}, 
\begin{eqnarray*} 
\frac{\psi^\Phi_\delta(x)- \psi(x)}{\delta^{2/(n+2)}} \leq  
\left(\frac{2^{n+1} (n+1)}{\vol_{n} \left(B^n_2\right)}\right)^\frac{2}{n+2} 
\frac{(1+\|\nabla \psi (x)\|^2)^{\frac{1}{2}}} { \eta ^\frac{2}{n+2}\, r_\psi(x) ^\frac{n}{n+2}}.
\end{eqnarray*}
Next, we treat the case
\begin{equation}\label{fall2}
0 < \psi^\Phi_\delta(x)- \psi(x)  < r_\psi(x)\  \langle e_{n+1}, N_{\psi} (z_x) \rangle.
\end{equation}
For all hyperplanes $H(z^\Phi_\delta)$ through $z^\Phi_\delta$ such that $z_x \in H^-(z^\Phi_\delta)$, 
$$ \vol_{n+1} \left(H^-(z^\Phi_\delta) \cap B^{n+1}_2 \left(z_x-r_\psi(x) N_\psi(z_x) , r_\psi(x) \right) \right)$$ 
is minimal if the line segment $[z^\Phi_\delta, z_x-r_\psi(x) N_\psi(z_x)]$ is orthogonal to the hyperplane $H(z^\Phi_\delta)$. 
Then 
$H^-(z^\Phi_\delta) \cap B^{n+1}_2 \left(z_x-r(x) N_\psi(z_x) , r_\psi(x) \right)$ is a cap of $B^{n+1}_2 \left(z_x-r(x) N_\psi(z_x) , r_\psi(x) \right)$ of height $d$, where
$d = \text{dist} \left( z^\Phi_\delta, \partial B^{n+1}_2 (z_x - r_\psi(x) N_{\psi} (z_x), r_\psi(x) \right)$. 
Let $h$ be the height of the cap $H_0^-(z^\Phi_\delta) \cap B^{n+1}_2 \left(z_x-r(x) N_\psi(z_x) , r_\psi(x) \right)$ and let $\beta$ be the angle between the normal to $H_0$ and the line segment 
$[z^\Phi_\delta, z_x-r_\psi(x) N_\psi(z_x)]$. 
If $\beta =0$, then $d=h = \psi^\Phi_\delta(x) -\psi(x)$ and we get, as above,
$$
\delta \geq  \eta  \frac{ \vol_n(B^n_2) }{2^\frac{n}{2}(n+2)}  \   \left(  \psi^\Phi_\delta(x) -\psi(x) \right) ^\frac{n+2}{2} \   r_\psi(x)^\frac{n}{2} 
$$
and thus 
\begin{eqnarray*} 
\frac{\psi^\Phi_\delta(x)- \psi(x)}{\delta^{2/(n+2)}} \leq 2^\frac{n}{n+2}
\left(\frac{ n+2}{\vol_{n} \left(B^n_2\right)}\right)^\frac{2}{n+2} 
\, \eta^{-\frac{2}{n+2}}\,     r_\psi(x) ^{-\frac{n}{n+2}}.
\end{eqnarray*}
Assume now that \( \beta>0\). We first consider the case  $h < r_\psi(x)$. Then
$$
\cos\beta = \frac{r_\psi(x)-h}{r_\psi(x)-d}    \hskip 5mm \text{and} \hskip 5mm \sin \beta = \frac{r_\psi(x) \ \langle e_{n+1}, N_{\psi} (z_x) \rangle - (\psi^\Phi_\delta(x) -\psi(x))}{r_\psi(x)-d}.
$$
From this we get 
\begin{eqnarray*}
d &=& r_\psi(x) - \left( (r_\psi(x)-h)^2 + \left(r_\psi(x) \ \langle e_{n+1}, N_{\psi} (z_x) \rangle - (\psi_\delta(x) -\psi(x))\right)^2\right)^\frac{1}{2}\\
&\geq &  r_\psi(x) \left( 1- \left(1 + \frac{(\psi_\delta(x) -\psi(x))^2}{r_\psi(x)^2}  - 2 \langle e_{n+1}, N_{\psi} (z_x) \rangle \frac{(\psi_\delta(x) -\psi(x))}{r_\psi(x)}  \right)^\frac{1}{2} \right)\\
&\geq &  \langle e_{n+1}, N_{\psi} (z_x) \rangle \left(\psi_\delta(x) -\psi(x)\right) \left(1 -  \frac{ \psi_\delta(x) -\psi(x) }{  2 \  r(x) \   \langle e_{n+1}, N_{\psi} (z_x) \rangle }\right) \\
&\geq&   \frac{1}{2} \langle e_{n+1}, N_{\psi} (z_x) \rangle \left(\psi_\delta(x) -\psi(x)\right).
\end{eqnarray*}
The latter inequality holds as  $\psi^\Phi_\delta(x)- \psi(x)  < r_\psi(x)\  \langle e_{n+1}, N_{\psi} (z_x) \rangle$. 
Thus we get with Lemma \ref{lemma:cap},
\begin{eqnarray*} 
\delta &\geq& \eta \,  \frac{ \vol_n(B^n_2) }{2^\frac{n}{2}(n+2)}  \   d^\frac{n+2}{2} \,  r_\psi(x)^\frac{n}{2} \\
&\geq& \eta \,  \frac{ \vol_n(B^n_2) }{2^{n+1} (n+2) }   \langle e_{n+1}, N_{\psi} (z_x) \rangle^\frac{n+2}{2}\   \left(\psi^\Phi_\delta(x) -\psi(x) \right) ^\frac{n+2}{2}\,    r_\psi(x)^\frac{n}{2}, 
\end{eqnarray*}
which implies that 
\begin{eqnarray*} 
\frac{\psi^\Phi_\delta(x)- \psi(x)}{\delta^{2/(n+2)}} \leq 2^{2   \frac{n+1}{n+2}}
\left(\frac{ n+2}{\vol_{n} \left(B^n_2\right)}\right)^\frac{2}{n+2} 
\   \frac{(1+\|\nabla \psi (x)\|^2)^{\frac{1}{2}}} {\eta^{\frac{2}{n+2}}\,  r_\psi(x) ^\frac{n}{n+2}}.
\end{eqnarray*}
If $h > r_\psi(x)$, then $\sin \beta$ is as above and $\cos\beta = \frac{h-r_\psi(x)}{r_\psi(x)-d}$. 
Continuing from here on as above, completes the proof of the lemma.
\end{proof}
 
We cannot apply the previous lemma to the exponential weight function $\Phi_e$, since this function is not uniformly bounded away from zero. We  prove the corresponding lemma for $\Phi_e$ separately.
We also need  a modified rolling function 
\begin{eqnarray}\label{cases}
		\rho_\psi(x) = 
		 \begin{cases} 
			r_\psi(x),   \hskip 10mm \text{if} \hskip 10mm 0 \leq r_\psi(x)  \leq \psi(x) \nonumber\\
			\psi(x),   \hskip 11mm \text{if}  \hskip 10mm r_\psi(x) > \psi(x) \geq 0 \nonumber \\
	r_\psi(x),   \hskip 10mm \text{if} \hskip 10mm 0 \leq r_\psi(x)  \leq -\psi(x)\nonumber\\
			-\psi(x),   \hskip 9mm \text{if}  \hskip 10mm r_\psi(x) > -\psi(x)>0.\\	
		\end{cases}\\ 
\end{eqnarray}
We recall that the set where $r_\psi(x)=0$ has measure $0$.

 \begin{lemma}\label{lemma:interchange2}
 Let  $\psi \in \text{Con}(\mathbb{R}^n)$.  
There exists $\delta_0$ such that for all $\delta < \delta_0$,  for  all $x \in \mathbb R^{n}$, 
\begin{align*}
0 &\leq \frac{\psi^e_\delta(x)- \psi(x)}{\delta^{2/(n+2)}} \, e^{-\psi(x)} \\
&\leq 2^\frac{3n+4}{n+2}
\, c_{n+1}  
 \begin{cases}  
 \frac{(1+\|\nabla \psi (x)\|^2)^{\frac{1}{2}}} { r_\psi(x) ^\frac{n}{n+2}} \, e^{-\frac{1}{n+2}(n \psi(x)-4 r_\psi(x))},  \hskip 3mm \text{if} \hskip 2mm |\psi(x)| \leq 1 \\
 \frac{(1+\|\nabla \psi (x)\|^2)^{\frac{1}{2}}} { \rho_\psi(x) ^\frac{n}{n+2}} \, e^{-\frac{1}{n+2}(n \psi(x)-4 \rho_\psi(x))},  \hskip 3mm \text{if}  \hskip 2mm |\psi(x)| >1
\end{cases}
\end{align*}
and
\begin{eqnarray*}
0 \leq  \frac{ f(x) - f^e_\delta(x) } {\delta^{2/(n+2)}} 
\leq  2^\frac{3n+4}{n+2}
\, c_{n+1} 
 \begin{cases} 
 \frac{(1+\|\nabla \psi (x)\|^2)^{\frac{1}{2}}} { r_\psi(x) ^\frac{n}{n+2}} \, e^{\frac{-n}{n+2}\left(\psi(x)-\frac{4}{n}r_\psi(x)\right)}, \hskip 3mm \text{if} \hskip 2mm |\psi(x)| \leq 1\\
\frac{(1+\|\nabla \psi (x)\|^2)^{\frac{1}{2}}} { \rho_\psi(x) ^\frac{n}{n+2}} \, e^{\frac{-n}{n+2}\left(\psi(x)-\frac{4}{n}\rho_\psi(x)\right)}, \hskip 3mm \text{if}  \hskip 2mm |\psi(x)| >1.
\end{cases}
\end{eqnarray*}
\end{lemma}
\begin{proof}  
We start out as in the proof of Lemma \ref{lemma:interchange0}. 
Let $z_x = (x, \psi(x)) \in \partial\left(\operatorname{epi}(\psi)\right)$ and  $z^e_\delta =(x, \psi^e_\delta(x))$.
Let first $x$ be such that $|\psi(x)| \leq 1$. Then, as above, 
\begin{eqnarray*}\label{interchange1}
\delta &=& \int_{H_\delta^- \cap \operatorname{epi}(\psi)} \Phi_e(z) dz \geq \int_{H_\delta^- \cap B^{n+1}_2 \left( z_x-r_\psi(x) N_\psi(z_x), r_\psi(x) \right)} \Phi_e(z) dz \nonumber\\
 &\geq & \min_{z \in H_\delta^- \cap B^{n+1}_2 \left( z_x-r_\psi(x) N_\psi(z_x), r_\psi(x) \right)} \Phi_e(z)\,  \vol_{n+1}\left(H_\delta^- \cap B^{n+1}_2 \left( z_x-r_\psi(x) N_\psi(z_x), r_\psi(x) \right) \right).
\end{eqnarray*}
Now we observe that 
\begin{eqnarray*}
\min_{z \in H_\delta^- \cap B^{n+1}_2 \left( z_x-r_\psi(x) N_\psi(z_x), r_\psi(x) \right)} \Phi_e(z) \geq e^{-(\psi(x) + 2 r_\psi(x))}
\end{eqnarray*}
and the  proof then continues as in Lemma \ref{lemma:interchange0}. 
This way, we arrive at 
\begin{eqnarray*} \label{estimate1}
\frac{\psi^e_\delta(x)- \psi(x)}{\delta^{2/(n+2)}} &\leq& 2^{2   \frac{n+1}{n+2}}
\left(\frac{ n+2}{\vol_{n} \left(B^n_2\right)}\right)^\frac{2}{n+2} 
\,   \frac{(1+\|\nabla \psi (x)\|^2)^{\frac{1}{2}}} {r_\psi(x) ^\frac{n}{n+2}}  e^{\frac{2}{n+2} (\psi(x) + 2 r_\psi(x))} \nonumber\\
&=& 2^\frac{3n+4}{n+2} c_{n+1} \, \frac{(1+\|\nabla \psi (x)\|^2)^{\frac{1}{2}}} { r_\psi(x) ^\frac{n}{n+2}}  e^{\frac{2}{n+2} (\psi(x) + 2 r_\psi(x))}
\end{eqnarray*}
and thus 
\begin{eqnarray*}
&&\frac{\psi^e_\delta(x)- \psi(x)}{\delta^{2/(n+2)}} \, e^{-\psi(x)}
\leq
2^\frac{3n+4}{n+2} c_{n+1} \, \frac{(1+\|\nabla \psi (x)\|^2)^{\frac{1}{2}}} { r_\psi(x) ^\frac{n}{n+2}}  e^{\frac{-1}{n+2} (n \psi(x) - 4 r_\psi(x))}.
\end{eqnarray*}
The term $ \frac{ f(x) - f^e_\delta(x) } {\delta^{2/(n+2)}}$ can be handled in the same way.
 
Let now $x$ be such that $|\psi(x)| > 1$. Then we use $\rho_\psi(x)$ and get that 
 \begin{eqnarray*}\label{interchange1}
\delta &=& \int_{H_\delta^- \cap \operatorname{epi}(\psi)} \Phi_e(z) dz \geq \int_{H_\delta^- \cap B^{n+1}_2 \left( z_x-\rho_\psi(x) N_\psi(z_x), \rho_\psi(x) \right)} \Phi_e(z) dz \nonumber\\
 &\geq & \min_{z \in H_\delta^- \cap B^{n+1}_2 \left( z_x-\rho_\psi(x) N_\psi(z_x), \rho_\psi(x) \right)} \Phi_e(z)\,  \vol_{n+1}\left(H_\delta^- \cap B^{n+1}_2 \left( z_x-\rho_\psi(x) N_\psi(z_x), \rho_\psi(x) \right) \right)
\end{eqnarray*}
and
\begin{eqnarray*}
\min_{z \in H_\delta^- \cap B^{n+1}_2 \left( z_x-\rho_\psi(x) N_\psi(z_x), \rho_\psi(x) \right)} \Phi_e(z) \geq e^{-(\psi(x) + 2 \rho_\psi(x))}.
\end{eqnarray*}
Again, the  proof then continues as in Lemma \ref{lemma:interchange0} and we obtain
\begin{eqnarray} \label{estimate1}
\frac{\psi^e_\delta(x)- \psi(x)}{\delta^{2/(n+2)}} &\leq& 2^{2   \frac{n+1}{n+2}}
\left(\frac{ n+2}{\vol_{n} \left(B^n_2\right)}\right)^\frac{2}{n+2} 
\,   \frac{(1+\|\nabla \psi (x)\|^2)^{\frac{1}{2}}} { \rho_\psi(x) ^\frac{n}{n+2}}  e^{\frac{2}{n+2} (\psi(x) + 2 \rho_\psi(x))} \nonumber\\
&=&2^\frac{3n+4}{n+2} c_{n+1} \, \frac{(1+\|\nabla \psi (x)\|^2)^{\frac{1}{2}}} { \rho_\psi(x) ^\frac{n}{n+2}}  e^{\frac{2}{n+2} (\psi(x) + 2 \rho_\psi(x))}.
\end{eqnarray}
Again, the term $ \frac{ f(x) - f^e_\delta(x) } {\delta^{2/(n+2)}}$ can be dealt with in the same way.
Thus,
\begin{eqnarray*}
&&\frac{\psi^e_\delta(x)- \psi(x)}{\delta^{2/(n+2)}} \, e^{-\psi(x)}
\leq
2^\frac{3n+4}{n+2} c_{n+1} \, \frac{(1+\|\nabla \psi (x)\|^2)^{\frac{1}{2}}} { \rho_\psi(x) ^\frac{n}{n+2}}  e^{\frac{-1}{n+2} (n \psi(x) - 4 \rho_\psi(x))}
\end{eqnarray*}
and similarly for $ \frac{ f(x) - f^e_\delta(x) } {\delta^{2/(n+2)}}$.
\end{proof}

We also need the following lemma which was proved in \cite{LiSchuettWerner}.
 
\begin{lemma}\label{integral1}
 Let  $\psi: \R^n \rightarrow \R $ be a  convex function such that  
$ e^{-\psi}$ is integrable. Then we have for all $0 \leq \alpha < 1$, 
\begin{equation}\label{integral-1}
\int _{\mathbb R^{n}} \frac{(1+\|\nabla \psi (x)\|^2)^{\frac{1}{2}}} { r_\psi(x) ^\alpha} \  e^{-\psi(x)} \   dx  < \infty.
\end{equation}
In particular, this holds for $\alpha = \frac{n}{n+2}$. 
\end{lemma}

\begin{remark}
If we analyze the proof of the lemma given in \cite{LiSchuettWerner}, we see  that in fact we have 
\begin{equation}\label{integral-1}
	\int _{\mathbb R^{n}} \frac{(1+\|\nabla \psi (x)\|^2)^{\frac{1}{2}}} { r_\psi(x) ^\alpha} \  e^{- \beta \psi(x)} \   dx  < \infty
\end{equation}
for all $0\leq \alpha <1$ and $ \beta>0$.
\end{remark}

For the proof of Theorem \ref{theo:f-deltafloat1} and Proposition \ref{corollary2}, we need a refinement of the previous lemma which we present next.
 
\begin{lemma}\label{integral2}
 Let  $\psi: \R^n \rightarrow \R $ be a  convex function such that  
$ e^{-\psi}$ is integrable. Then we have for all $0 \leq \alpha < 1$, 
 \begin{equation*}\label{integral-1}
\int _{\{x: |\psi(x) |> 1\}} \frac{(1+\|\nabla \psi (x)\|^2)^{\frac{1}{2}}} { \rho_\psi(x) ^\alpha} \,   e^{\frac{-1}{n+2}\left(n \psi(x)- 4 \rho_\psi(x)\right)} \   dx  < \infty
\end{equation*}
and
\begin{equation*}\label{integral-1}
\int _{\{x: |\psi(x)| \leq 1\}} \frac{(1+\|\nabla \psi (x)\|^2)^{\frac{1}{2}}} { r_\psi(x) ^\alpha} \,   e^{\frac{-1}{n+2}\left(n \psi(x)- 4 r_\psi(x)\right)} \   dx  < \infty.
\end{equation*}
In particular, this  holds for $\alpha = \frac{n}{n+2}$. 
\end{lemma}
\begin{proof}  
We start by writing	
\begin{eqnarray*}
&&\hskip -15mm \int _{\{x: |\psi(x) |> 1\}} \frac{(1+\|\nabla \psi (x)\|^2)^{\frac{1}{2}}} { \rho_\psi(x) ^\alpha} \,   e^{\frac{-1}{n+2}\left(n \psi(x)- 4 \rho_\psi(x)\right)} \,   dx  \\
&& = \int _{\{x:\,  \psi>1 \, \text{and}\,  0 < r_\psi \leq \psi\} } \frac{(1+\|\nabla \psi (x)\|^2)^{\frac{1}{2}}} { \rho_\psi(x) ^\alpha} \,   e^{\frac{-1}{n+2}\left(n \psi(x)- 4 \rho_\psi(x)\right)} \,  dx  \\
&& + \int _{ \{x:\, \psi >1 \, \text{and}\, 1< \psi < r_\psi\}} \frac{(1+\|\nabla \psi (x)\|^2)^{\frac{1}{2}}} { \rho_\psi(x) ^\alpha} \,   e^{\frac{-1}{n+2}\left(n \psi(x)- 4 \rho_\psi(x)\right)} \,   dx\\
&& + \int _{\{x:\,  \psi < -1 \, \text{and}\,  0 < r_\psi \leq -\psi\} } \frac{(1+\|\nabla \psi (x)\|^2)^{\frac{1}{2}}} { \rho_\psi(x) ^\alpha} \,   e^{\frac{-1}{n+2}\left(n \psi(x)- 4 \rho_\psi(x)\right)} \,  dx  \\
&& + \int _{ \{x:\, \psi <-1 \, \text{and}\, 1 < -\psi < r_\psi\}} \frac{(1+\|\nabla \psi (x)\|^2)^{\frac{1}{2}}} { \rho_\psi(x) ^\alpha} \,   e^{\frac{-1}{n+2}\left(n \psi(x)- 4 \rho_\psi(x)\right)} \,   dx.
\end{eqnarray*}
Now, we estimate  these integrals further according to the different cases indicated in \eqref{cases} and get 
 \begin{align*}
&\int _{\{x: |\psi(x) |> 1\}} \frac{(1+\|\nabla \psi (x)\|^2)^{\frac{1}{2}}} { \rho_\psi(x) ^\alpha} \,   e^{\frac{-1}{n+2}\left(n \psi(x)- 4 \rho_\psi(x)\right)} \,   dx   \\ 
&\leq\int _{\{x:\,  \psi>1 \, \text{and}\,  0 < r_\psi \leq \psi\} } \frac{(1+\|\nabla \psi (x)\|^2)^{\frac{1}{2}}} { r_\psi(x) ^\alpha} \,  e^{-\frac{n-4}{n+2} \psi(x)}\, dx	   \\
&\qquad+ 	\int _{ \{x:\, \psi >1 \, \text{and}\, 1< \psi < r_\psi\}}		 \frac{(1+\|\nabla \psi (x)\|^2)^{\frac{1}{2}}} { \psi(x) ^\alpha} \,  e^{-\frac{n-4}{n+2} \psi(x)} \, dx\\
&\qquad +  \int _{\{x:\,  \psi <-1 \, \text{and}\,  0 < r_\psi \leq -\psi\} } \frac{(1+\|\nabla \psi (x)\|^2)^{\frac{1}{2}}} { r_\psi(x) ^\alpha} \,  e^{-\frac{n+4}{n+2} \psi(x)} \, dx\\
&\qquad + \int _{ \{x:\, \psi <-1 \, \text{and}\, 1 <  -\psi < r_\psi\}}  \frac{(1+\|\nabla \psi (x)\|^2)^{\frac{1}{2}}} { (-\psi(x)) ^\alpha} \,  e^{-\frac{n+4}{n+2} \psi(x)} \, dx\\
& \leq  \int _{\mathbb R^{n}} \frac{(1+\|\nabla \psi (x)\|^2)^{\frac{1}{2}}} { r_\psi(x) ^\alpha} \,  e^{-\frac{n-4}{n+2} \psi(x)}  \, dx + \int _{\mathbb R^{n}} \frac{(1+\|\nabla \psi (x)\|^2)^{\frac{1}{2}}} { r_\psi(x) ^\alpha} \,  e^{-\frac{n+4}{n+2} \psi(x)}  \, dx\\
&\qquad + \int _{ \mathbb{R}^n} (1+\|\nabla \psi (x)\|^2)^{\frac{1}{2}} \,  e^{-\frac{n-4}{n+2} \psi(x)}  \, dx \, +\,  \int _{ \mathbb{R}^n} (1+\|\nabla \psi (x)\|^2)^{\frac{1}{2}}  \,  e^{-\frac{n+4}{n+2} \psi(x)}  \, dx.
\end{align*}
The first two integrals in the last inequality are finite by Lemma \ref{integral1} and the subsequent remark.
It was proved in Lemma 3 of  \cite{LiSchuettWerner} that the last  two integrals are finite as well.

Now we consider the second integral
\begin{equation*}\label{integral-1}
\int _{\{x: |\psi(x)| \leq 1\}} \frac{(1+\|\nabla \psi (x)\|^2)^{\frac{1}{2}}} { r_\psi(x) ^\alpha} \,   e^{\frac{-1}{n+2}\left(n \psi(x)- 4 r_\psi(x)\right)} \   dx.
\end{equation*}
The set $\{x: |\psi(x)| \leq 1\}$ is compact. Indeed, since $\psi$ is continuous it is closed and since $\int_{\mathbb{R}^n} e^{-\psi(x)} dx < \infty$, the set is also bounded.
It is shown in \cite{SchuettWernerBook} that the rolling function $r_\psi$ is upper  semicontinuous.  Therefore the upper semicontinuous function $e^{\frac{4}{n+2}  r_\psi(x)}$ attains its maximum on $\{x: |\psi(x)| \leq 1\}$. Thus,
\begin{eqnarray*}
&&\int _{\{x: |\psi(x)| \leq 1\}} \frac{(1+\|\nabla \psi (x)\|^2)^{\frac{1}{2}}} { r_\psi(x) ^\alpha} \,   e^{\frac{-1}{n+2}\left(n \psi(x)- 4 r_\psi(x)\right)} \   dx \\
&\leq &\max_{\{x: |\psi(x)| \leq 1\}}  e^{\frac{4}{n+2}  r_\psi(x)} \int _{\{x: |\psi(x)| \leq 1\}} \frac{(1+\|\nabla \psi (x)\|^2)^{\frac{1}{2}}} { r_\psi(x) ^\alpha} \,   e^{\frac{-n}{n+2} \psi(x)} \,  dx \\
&\leq &\max_{\{x: |\psi(x)| \leq 1\}}  e^{\frac{4}{n+2}  r_\psi(x)} \int _{\mathbb{R}^n} \frac{(1+\|\nabla \psi (x)\|^2)^{\frac{1}{2}}} { r_\psi(x) ^\alpha} \,   e^{\frac{-n}{n+2} \psi(x)}\, dx,
\end{eqnarray*}
and the latter integral is finite by Lemma \ref{integral1} and again the remark made thereafter.
\end{proof}

Theorem  \ref{theo:f-deltafloat2} and Theorem  \ref{theo:f-deltafloat1} as well as Proposition \ref{corollary1}  and Proposition \ref{corollary2} follow immediately from the lemmas just presented.

\begin{proof}[Proof of Theorem \ref{theo:f-deltafloat2} and Theorem \ref{theo:f-deltafloat1}]
By the assumptions of the theorems, Lemmas \ref{lemma:interchange0} respectively \ref{lemma:interchange2},  Lemmas \ref{integral1}  respectively \ref{integral2},  and the dominated convergence theorem we get  
\begin{eqnarray*}
\lim _{\delta \rightarrow 0} \frac{ \int_{\R^n} (f (x) - f^\Phi_\delta(x) )  dx} {\delta^{2/(n+2)}} =
 \int_{\R^n} \lim _{\delta \rightarrow 0} \frac{(f(x) - f^\Phi_\delta(x) ) \ dx} {\delta^{2/(n+2)}} .
\end{eqnarray*}
and
\begin{eqnarray*}
\lim _{\delta \rightarrow 0} \frac{ \int _{\mathbb R^{n}}(f(x)  - f_{\delta}^{e}(x) )  \  dx } {\delta^{2/(n+2)}}  =  \int_{\R^n} \lim _{\delta \rightarrow 0} \frac{(f(x) - f^e_\delta(x) ) \ dx} {\delta^{2/(n+2)}}.
\end{eqnarray*}
Finally, Lemma \ref{lemma:f-delta} finishes the proof.
\end{proof}

\begin{proof}[Proof of Proposition  \ref{corollary1} and Proposition \ref{corollary2}]
The proofs are done in the same way as for the theorems.
\end{proof}
 
\subsection{Proof of Theorem \ref{thm:sconcave}}

Recall the definition of the convex body $K_f^s$ and the fact that
$$
\vol_{n+s}(K_f^s)=\int_{\R^n}\vol_s(B_2^s) \bigl(f^{1/s}(x)\bigr)^s\, dx=\vol_s(B_2^s)\int_{\R^n}f(x)\, dx.
$$
Then we can write
\[
\int_{\R^n} f(x)-f^\Phi_\delta(x)\, dx=\frac{\vol_{n+s}(K_f^s)-\vol_{n+s}(K_f^s(\Phi,\delta))}{\vol_s(B_2^s)}.
\]
Now we apply Theorem 5 of \cite{Werner2002} for the volume difference involving the weighted floating body.  It follows that
\begin{align*}
&\lim_{\delta\to 0} \frac{\vol(K_f^s)-\vol(K_f^s(\Phi,\delta))}{\delta^{2/(n+s+1)}}\\
&={1\over 2}\Bigl( \frac{n+s+1}{(n+s)\vol_{n+s}(B_2^{n+s})}\Bigr)^{\frac{2}{n+s+1}}\int_{\partial K_f^s}\kappa(z)^{{1\over n+s+1}}\Phi(z)^{-\frac{2}{n+s+1}}\,\mu_{K_f^s}(dz).
\end{align*}
For the next lemma recall the definition of \(\mathrm{as}^s_\Phi(f)\) from \eqref{def:assphi}.
\begin{lemma}
	One has that
	\[
		\int_{\partial K_f^s}\kappa(z)^{{1\over n+s+1}}\Phi(z)^{-\frac{2}{n+s+1}}\,\mu_{K_f^s}(dz)=s\,\vol_s(B_2^s)\,\mathrm{as}^s_\Phi(f)
	\]
\end{lemma}
\begin{proof}
	We have
	\[
		\int_{\partial K_f^s}\kappa(z)^{1\over n+s+1}\Phi(z)^{-{2\over n+s+1}}\,\mu_{K_f^s}(d z)
		=\int_{\tilde\partial K_f^s}\kappa(z)^{1\over n+s+1}\Phi(z)^{-{2\over n+s+1}}\,\mu_{K_f^s}(d z),
	\]
	where $\tilde\partial K_f^s=\{(x,y)\in\partial K_f^s:x\in{\rm int}({\rm supp}(f))\}$ is the set of points on the boundary whose projection onto the $\R^n$-coordinate belongs to the interior of the support of $f$. Now, applying \cite[Lemma 8]{AAKSW} and writing $z=(z_1,\ldots,z_n,z_{n+1},\ldots,z_{n+s})\in\R^{n+s}$, we arrive at
	\begin{align*}
		\int_{\tilde\partial K_f^s}&\kappa(z)^{1\over n+s+1}\Phi(z)^{-{2\over n+s+1}}\,\mu_{K_f^s}(d z)\\
		&= \int_{\tilde\partial K_f^s}{|\det\nabla^2 f(z_1,\ldots,z_n)^{1/s}|^{1\over n+s+1}\over\sqrt{1+\|\nabla f(z_1,\ldots,z_n)^{1/s}\|^2}} f(z_1,\ldots,z_n)^{-{s-1\over s(n+s+1)}}\Phi(z)^{-{2\over n+s+1}}\,\mu_{K_f^s}(d z)\\
		&= 2\int_{\R^n}\int_{\R^{s-1}}f(x)^{1\over s}\Big({|\det\nabla^2f(x)^{1/s}|\over f(x)^{s-1\over s}}\Big)^{1\over n+s+1}\phi(x,f(x)^{1/s})^{-{2\over n+s+1}}{d y\over|y_s|}d x,
	\end{align*}
	where $y=(y_1,\ldots,y_{s-1})$ and $y_s:=\sqrt{f(x)^{2/s}-\sum_{i=1}^{s-1}y_i^2}$. The integral with respect to $y$ is the same as in \cite{AAKSW} and equals
	\[
	\int_{\R^{s-1}}{d y\over|y_s|} = {1\over 2}{f(x)^{s-1\over s}\over f(x)^{1\over s}}(s-1)\vol_{s-1}(B_2^{s-1})B\Big({s-1\over 2},{1\over 2}\Big)
	\]
	with Euler's beta function $B(\,\cdot\,,\,\cdot\,)$. As a consequence,
	\begin{align*}
		&\int_{\partial K_f^s}\kappa(z)^{1\over n+s+1}\Phi(z)^{-{2\over n+s+1}}\,\mu_{K_f^s}(d z)\\
		&= (s-1)\vol_{s-1}(B_2^{s-1})B\Bigl({s-1\over 2},{1\over 2}\Bigr)\\
		&\qquad\qquad\times\int_{\R^n}|\det \nabla^2f(x)^{1/s}|^{{1\over n+s+1}}f(x)^{\frac{(s-1)(n+s)}{s(n+s+1)}}\psi(x,f^{1/s}(x))^{-\frac{2}{n+s+1}}\,d x.
	\end{align*}
	The result now follows from the fact that
	\begin{align*}
		(s-1)\vol_{s-1}(B_2^{s-1})B\Big({s-1\over 2},{1\over 2}\Big) &= (s-1){\pi^{s-1\over 2}\over\Gamma({s-1\over 2}+1)}{\Gamma({s-1\over 2})\Gamma({1\over 2})\over\Gamma({s\over 2})}\\
		&= (s-1){\pi^{s-1\over 2}\over {s-1\over 2}\Gamma({s-1\over 2})}{\Gamma({s-1\over 2})\sqrt{\pi}\over\Gamma({s\over 2})} = {2\pi^{s/2}\over\Gamma({s\over 2})}\\
		& = s\,\vol_s(B_2^s).
	\end{align*}
	The proof is thus complete.
\end{proof}

We are now prepared for the proof of Theorem \ref{thm:sconcave}.

\begin{proof}
We have
\[
\begin{split}
	&\lim_{\delta\to 0} \frac{\int_{\R^n} f(x)-f^\Phi_\delta(x)\,d x}{\delta^{2/(n+s+1)}} \\
	&={1\over 2\vol_s(B_2^s)}\Bigl( \frac{n+s+1}{(n+s)\vol_s(B_2^{n+s})}\Bigr)^{\frac{2}{n+s+1}}  s\,\vol_s(B_2^s)\,\mathrm{as}^s_\Phi(f)\\
	&={s\,\vol_s(B_2^s)\over 2\vol_s(B_2^s)}\Bigl( \frac{n+s+1}{(n+s)\vol_{n+s}(B_2^{n+s})}\Bigr)^{\frac{2}{n+s+1}}\\
	&\qquad\qquad\times\int_{\R^n}|\det \nabla^2f(x)^{1/s}|^{{1\over n+s+1}}f(x)^{\frac{(s-1)(n+s)}{s(n+s+1)}}\phi(x,f^{1/s}(x))^{-\frac{2}{n+s+1}}\,d x\\
	&={s\over 2}\Bigl( \frac{n+s+1}{(n+s)\vol_{n+s}(B_2^{n+s})}\Bigr)^{\frac{2}{n+s+1}}\\
	&\qquad\qquad\times\int_{\R^n}|\det \nabla^2f(x)^{1/s}|^{{1\over n+s+1}}f(x)^{\frac{(s-1)(n+s)}{s(n+s+1)}}\phi(x,f^{1/s}(x))^{-\frac{2}{n+s+1}}\,d x\\
	& =c_{n,s}\int_{\R^n}|\det \nabla^2f(x)^{1/s}|^{{1\over n+s+1}}f(x)^{\frac{(s-1)(n+s)}{s(n+s+1)}}\phi(x,f^{1/s}(x))^{-\frac{2}{n+s+1}}\,d x.
\end{split}
\]
This completes the proof of Theorem \ref{thm:sconcave}. 
\end{proof}

\subsection*{Acknowledgement}
CT has  been supported by the DFG priority program SPP 2265 \textit{Random Geometric Systems}. NT has been supported by the INdAM-GNAMPA Project CUP E55F22000270001. EMW is supported by NSF grant  \textit{DMS-2103482} and by the \textit{Caroline Herschel Guest Professorship} of the Ruhr University Bochum.

 {}

\vskip 4mm
 
Carsten Sch\"utt\\
{\small Mathematisches Seminar }\\
{\small Christian-Albrechts Universit\"at}\\
{\small 24098 Kiel, Germany }\\
{\small \tt schuett@math.uni-kiel.de}
\vskip 3mm
 
Christoph Th\"ale\\
\small{Faculty of Mathematics}\\
{\small Ruhr University Bochum}\\ 
{\small Bochum, Germany}\\ 
{\small \tt christoph.thaele@rub.de}
\vskip 3mm

Nicola Turchi\\
\small{Department of Mathematics and Applications}\\
{\small University of Milano-Bicocca}\\ 
{\small 20125 Milano, Italy}\\ 
{\small \tt nicola.turchi@unimib.it}
\vskip 3mm
 
Elisabeth M. Werner\\
{\small Department of Mathematics \ \ \ \ \ \ \ \ \ \ \ \ \ \ \ \ \ \ \ Universit\'{e} de Lille 1}\\
{\small Case Western Reserve University \ \ \ \ \ \ \ \ \ \ \ \ \ UFR de Math\'{e}matique }\\
{\small Cleveland, Ohio 44106, U. S. A. \ \ \ \ \ \ \ \ \ \ \ \ \ \ \ 59655 Villeneuve d'Ascq, France}\\
{\small \tt elisabeth.werner@case.edu}\\ \\

\end{document}